\documentclass[10pt,leqno]{amsart}
\usepackage{amssymb}
\usepackage{xypic}
\begin{document}
\theoremstyle{plain}
\newtheorem{thm}{Theorem}[section]
\newtheorem*{thm1}{Theorem 1}
\newtheorem*{thm2}{Theorem 2}
\newtheorem{lemma}[thm]{Lemma}
\newtheorem{lem}[thm]{Lemma}
\newtheorem{cor}[thm]{Corollary}
\newtheorem{prop}[thm]{Proposition}
\newtheorem{propose}[thm]{Proposition}
\newtheorem{variant}[thm]{Variant}
\theoremstyle{definition}
\newtheorem{notations}[thm]{Notations}
\newtheorem{rem}[thm]{Remark}
\newtheorem{rmk}[thm]{Remark}
\newtheorem{rmks}[thm]{Remarks}
\newtheorem{defn}[thm]{Definition}
\newtheorem{ex}[thm]{Example}
\newtheorem{claim}[thm]{Claim}
\newtheorem{ass}[thm]{Assumption}
\numberwithin{equation}{section}
\newcounter{elno}                
\def\points{\list
{\hss\llap{\upshape{(\roman{elno})}}}{\usecounter{elno}}} 
\let\endpoints=\endlist


\catcode`\@=11
%
%
\def\opn#1#2{\def#1{\mathop{\kern0pt\fam0#2}\nolimits}} 
\def\bold#1{{\bf #1}}%
\def\underrightarrow{\mathpalette\underrightarrow@}
\def\underrightarrow@#1#2{\vtop{\ialign{$##$\cr
 \hfil#1#2\hfil\cr\noalign{\nointerlineskip}%
 #1{-}\mkern-6mu\cleaders\hbox{$#1\mkern-2mu{-}\mkern-2mu$}\hfill
 \mkern-6mu{\to}\cr}}}
\let\underarrow\underrightarrow
\def\underleftarrow{\mathpalette\underleftarrow@}
\def\underleftarrow@#1#2{\vtop{\ialign{$##$\cr
 \hfil#1#2\hfil\cr\noalign{\nointerlineskip}#1{\leftarrow}\mkern-6mu
 \cleaders\hbox{$#1\mkern-2mu{-}\mkern-2mu$}\hfill
 \mkern-6mu{-}\cr}}}
%
%

%
\def\:{\colon}
\let\oldtilde=\tilde
\def\tilde#1{\mathchoice{\widetilde{#1}}{\widetilde{#1}}%
{\indextil{#1}}{\oldtilde{#1}}}
\def\indextil#1{\lower2pt\hbox{$\textstyle{\oldtilde{\raise2pt%
\hbox{$\scriptstyle{#1}$}}}$}}
\def\pnt{{\raise1.1pt\hbox{$\textstyle.$}}}
%

%
\let\amp@rs@nd@\relax
\newdimen\ex@\ex@.2326ex
\newdimen\bigaw@l
\newdimen\minaw@
\minaw@16.08739\ex@
\newdimen\minCDaw@
\minCDaw@2.5pc
\newif\ifCD@
\def\minCDarrowwidth#1{\minCDaw@#1}
\newenvironment{CD}{\@CD}{\@endCD}
\def\@CD{\def\A##1A##2A{\llap{$\vcenter{\hbox
 {$\scriptstyle##1$}}$}\Big\uparrow\rlap{$\vcenter{\hbox{%
$\scriptstyle##2$}}$}&&}%
\def\V##1V##2V{\llap{$\vcenter{\hbox
 {$\scriptstyle##1$}}$}\Big\downarrow\rlap{$\vcenter{\hbox{%
$\scriptstyle##2$}}$}&&}%
\def\={&\hskip.5em\mathrel
 {\vbox{\hrule width\minCDaw@\vskip3\ex@\hrule width
 \minCDaw@}}\hskip.5em&}%
\def\verteq{\Big\Vert&&}%
\def\noarr{&&}%
\def\vspace##1{\noalign{\vskip##1\relax}}\relax\let\amp@rs@nd@&\iffalse}\fi
 \CD@true\vcenter\bgroup\relax\let\\=\cr\iffalse}\fi\tabskip\z@skip\baselineskip20\ex@
 \lineskip3\ex@\lineskiplimit3\ex@\halign\bgroup
 &\hfill$\m@th##$\hfill\cr}
\def\@endCD{\cr\egroup\egroup}
%
\def\>#1>#2>{\amp@rs@nd@\setbox\z@\hbox{$\scriptstyle
 \;{#1}\;\;$}\setbox\@ne\hbox{$\scriptstyle\;{#2}\;\;$}\setbox\tw@
 \hbox{$#2$}\ifCD@
 \global\bigaw@\minCDaw@\else\global\bigaw@\minaw@\fi
 \ifdim\wd\z@>\bigaw@\global\bigaw@\wd\z@\fi
 \ifdim\wd\@ne>\bigaw@\global\bigaw@\wd\@ne\fi
 \ifCD@\hskip.5em\fi
 \ifdim\wd\tw@>\z@
 \mathrel{\mathop{\hbox to\bigaw@{\rightarrowfill}}\limits^{#1}_{#2}}\else
 \mathrel{\mathop{\hbox to\bigaw@{\rightarrowfill}}\limits^{#1}}\fi
 \ifCD@\hskip.5em\fi\amp@rs@nd@}
\def\<#1<#2<{\amp@rs@nd@\setbox\z@\hbox{$\scriptstyle
 \;\;{#1}\;$}\setbox\@ne\hbox{$\scriptstyle\;\;{#2}\;$}\setbox\tw@
 \hbox{$#2$}\ifCD@
 \global\bigaw@\minCDaw@\else\global\bigaw@\minaw@\fi
 \ifdim\wd\z@>\bigaw@\global\bigaw@\wd\z@\fi
 \ifdim\wd\@ne>\bigaw@\global\bigaw@\wd\@ne\fi
 \ifCD@\hskip.5em\fi
 \ifdim\wd\tw@>\z@
 \mathrel{\mathop{\hbox to\bigaw@{\leftarrowfill}}\limits^{#1}_{#2}}\else
 \mathrel{\mathop{\hbox to\bigaw@{\leftarrowfill}}\limits^{#1}}\fi
 \ifCD@\hskip.5em\fi\amp@rs@nd@}
%
%
\newenvironment{CDS}{\@CDS}{\@endCDS}
\def\@CDS{\def\A##1A##2A{\llap{$\vcenter{\hbox
 {$\scriptstyle##1$}}$}\Big\uparrow\rlap{$\vcenter{\hbox{%
$\scriptstyle##2$}}$}&}%
\def\V##1V##2V{\llap{$\vcenter{\hbox
 {$\scriptstyle##1$}}$}\Big\downarrow\rlap{$\vcenter{\hbox{%
$\scriptstyle##2$}}$}&}%
\def\={&\hskip.5em\mathrel
 {\vbox{\hrule width\minCDaw@\vskip3\ex@\hrule width
 \minCDaw@}}\hskip.5em&}
\def\verteq{\Big\Vert&}
\def\novarr{&}
\def\noharr{&&}
\def\SE##1E##2E{\slantedarrow(0,18)(4,-3){##1}{##2}&}
\def\SW##1W##2W{\slantedarrow(24,18)(-4,-3){##1}{##2}&}
\def\NE##1E##2E{\slantedarrow(0,0)(4,3){##1}{##2}&}
\def\NW##1W##2W{\slantedarrow(24,0)(-4,3){##1}{##2}&}
\def\slantedarrow(##1)(##2)##3##4{%
\thinlines\unitlength1pt\lower 6.5pt\hbox{\begin{picture}(24,18)%
\put(##1){\vector(##2){24}}%
\put(0,8){$\scriptstyle##3$}%
\put(20,8){$\scriptstyle##4$}%
\end{picture}}}
\def\vspace##1{\noalign{\vskip##1\relax}}\relax\let\amp@rs@nd@&\iffalse}\fi
 \CD@true\vcenter\bgroup\relax\let\\=\cr\iffalse}\fi\tabskip\z@skip\baselineskip20\ex@
 \lineskip3\ex@\lineskiplimit3\ex@\halign\bgroup
 &\hfill$\m@th##$\hfill\cr}
\def\@endCDS{\cr\egroup\egroup}
%
\newdimen\TriCDarrw@
\newif\ifTriV@
\newenvironment{TriCDV}{\@TriCDV}{\@endTriCD}
\newenvironment{TriCDA}{\@TriCDA}{\@endTriCD}
\def\@TriCDV{\TriV@true\def\TriCDpos@{6}\@TriCD}
\def\@TriCDA{\TriV@false\def\TriCDpos@{10}\@TriCD}
\def\@TriCD#1#2#3#4#5#6{%
\setbox0\hbox{$\ifTriV@#6\else#1\fi$}
\TriCDarrw@=\wd0 \advance\TriCDarrw@ 24pt
\advance\TriCDarrw@ -1em
\def\SE##1E##2E{\slantedarrow(0,18)(2,-3){##1}{##2}&}
\def\SW##1W##2W{\slantedarrow(12,18)(-2,-3){##1}{##2}&}
\def\NE##1E##2E{\slantedarrow(0,0)(2,3){##1}{##2}&}
\def\NW##1W##2W{\slantedarrow(12,0)(-2,3){##1}{##2}&}
\def\slantedarrow(##1)(##2)##3##4{\thinlines\unitlength1pt
\lower 6.5pt\hbox{\begin{picture}(12,18)%
\put(##1){\vector(##2){12}}%
\put(-4,\TriCDpos@){$\scriptstyle##3$}%
\put(12,\TriCDpos@){$\scriptstyle##4$}%
\end{picture}}}
\def\={\mathrel {\vbox{\hrule
   width\TriCDarrw@\vskip3\ex@\hrule width
   \TriCDarrw@}}}
\def\>##1>>{\setbox\z@\hbox{$\scriptstyle
 \;{##1}\;\;$}\global\bigaw@\TriCDarrw@
 \ifdim\wd\z@>\bigaw@\global\bigaw@\wd\z@\fi
 \hskip.5em
 \mathrel{\mathop{\hbox to \TriCDarrw@
{\rightarrowfill}}\limits^{##1}}
 \hskip.5em}
\def\<##1<<{\setbox\z@\hbox{$\scriptstyle
 \;{##1}\;\;$}\global\bigaw@\TriCDarrw@
 \ifdim\wd\z@>\bigaw@\global\bigaw@\wd\z@\fi
 \mathrel{\mathop{\hbox to\bigaw@{\leftarrowfill}}\limits^{##1}}
 }
 \CD@true\vcenter\bgroup\relax\let\\=\cr\iffalse}\fi
 \tabskip\z@skip\baselineskip20\ex@
 \lineskip3\ex@\lineskiplimit3\ex@
 \ifTriV@
 \halign\bgroup
 &\hfill$\m@th##$\hfill\cr
#1&\multispan3\hfill$#2$\hfill&#3\\
&#4&#5\\
&&#6\cr\egroup%
\else
 \halign\bgroup
 &\hfill$\m@th##$\hfill\cr
&&#1\\%
&#2&#3\\
#4&\multispan3\hfill$#5$\hfill&#6\cr\egroup
\fi}
\def\@endTriCD{\egroup} 
\newcommand{\mc}{\mathcal} 
\newcommand{\mb}{\mathbb} 
\newcommand{\surj}{\twoheadrightarrow} 
\newcommand{\inj}{\hookrightarrow} \newcommand{\zar}{{\rm zar}} 
\newcommand{\an}{{\rm an}} \newcommand{\red}{{\rm red}} 
\newcommand{\Rank}{{\rm rk}} \newcommand{\codim}{{\rm codim}} 
\newcommand{\rank}{{\rm rank}} \newcommand{\Ker}{{\rm Ker \ }} 
\newcommand{\Pic}{{\rm Pic}} \newcommand{\Div}{{\rm Div}} 
\newcommand{\Hom}{{\rm Hom}} \newcommand{\im}{{\rm im}} 
\newcommand{\Spec}{{\rm Spec \,}} \newcommand{\Sing}{{\rm Sing}} 
\newcommand{\sing}{{\rm sing}} \newcommand{\reg}{{\rm reg}} 
\newcommand{\Char}{{\rm char}} \newcommand{\Tr}{{\rm Tr}} 
\newcommand{\Gal}{{\rm Gal}} \newcommand{\Min}{{\rm Min \ }} 
\newcommand{\Max}{{\rm Max \ }} \newcommand{\Alb}{{\rm Alb}\,} 
\newcommand{\GL}{{\rm GL}\,} 
\newcommand{\ie}{{\it i.e.\/},\ } \newcommand{\niso}{\not\cong} 
\newcommand{\nin}{\not\in} 
\newcommand{\soplus}[1]{\stackrel{#1}{\oplus}} 
\newcommand{\by}[1]{\stackrel{#1}{\rightarrow}} 
\newcommand{\longby}[1]{\stackrel{#1}{\longrightarrow}} 
\newcommand{\vlongby}[1]{\stackrel{#1}{\mbox{\large{$\longrightarrow$}}}} 
\newcommand{\ldownarrow}{\mbox{\Large{\Large{$\downarrow$}}}} 
\newcommand{\lsearrow}{\mbox{\Large{$\searrow$}}} 
\renewcommand{\d}{\stackrel{\mbox{\scriptsize{$\bullet$}}}{}} 
\newcommand{\dlog}{{\rm dlog}\,} 
\newcommand{\longto}{\longrightarrow} 
\newcommand{\vlongto}{\mbox{{\Large{$\longto$}}}} 
\newcommand{\limdir}[1]{{\displaystyle{\mathop{\rm lim}_{\buildrel\longrightarrow\over{#1}}}}\,} 
\newcommand{\liminv}[1]{{\displaystyle{\mathop{\rm lim}_{\buildrel\longleftarrow\over{#1}}}}\,} 
\newcommand{\norm}[1]{\mbox{$\parallel{#1}\parallel$}} 
\newcommand{\boxtensor}{{\Box\kern-9.03pt\raise1.42pt\hbox{$\times$}}} 
\newcommand{\into}{\hookrightarrow} \newcommand{\image}{{\rm image}\,} 
\newcommand{\Lie}{{\rm Lie}\,} 
\newcommand{\CM}{\rm CM}
\newcommand{\sext}{\mbox{${\mathcal E}xt\,$}} 
\newcommand{\shom}{\mbox{${\mathcal H}om\,$}} 
\newcommand{\coker}{{\rm coker}\,} 
\newcommand{\sm}{{\rm sm}} 
\newcommand{\tensor}{\otimes} 
\renewcommand{\iff}{\mbox{ $\Longleftrightarrow$ }} 
\newcommand{\supp}{{\rm supp}\,} 
\newcommand{\ext}[1]{\stackrel{#1}{\wedge}} 
\newcommand{\onto}{\mbox{$\,\>>>\hspace{-.5cm}\to\hspace{.15cm}$}} 
\newcommand{\propsubset} {\mbox{$\textstyle{ 
\subseteq_{\kern-5pt\raise-1pt\hbox{\mbox{\tiny{$/$}}}}}$}} 
\newcommand{\sB}{{\mathcal B}} \newcommand{\sC}{{\mathcal C}} 
\newcommand{\sD}{{\mathcal D}} \newcommand{\sE}{{\mathcal E}} 
\newcommand{\sF}{{\mathcal F}} \newcommand{\sG}{{\mathcal G}} 
\newcommand{\sH}{{\mathcal H}} \newcommand{\sI}{{\mathcal I}} 
\newcommand{\sJ}{{\mathcal J}} \newcommand{\sK}{{\mathcal K}} 
\newcommand{\sL}{{\mathcal L}} \newcommand{\sM}{{\mathcal M}} 
\newcommand{\sN}{{\mathcal N}} \newcommand{\sO}{{\mathcal O}} 
\newcommand{\sP}{{\mathcal P}} \newcommand{\sQ}{{\mathcal Q}} 
\newcommand{\sR}{{\mathcal R}} \newcommand{\sS}{{\mathcal S}} 
\newcommand{\sT}{{\mathcal T}} \newcommand{\sU}{{\mathcal U}} 
\newcommand{\sV}{{\mathcal V}} \newcommand{\sW}{{\mathcal W}} 
\newcommand{\sX}{{\mathcal X}} \newcommand{\sY}{{\mathcal Y}} 
\newcommand{\sZ}{{\mathcal Z}} \newcommand{\ccL}{\sL} 
 \newcommand{\A}{{\mathbb A}} \newcommand{\B}{{\mathbb 
B}} \newcommand{\C}{{\mathbb C}} \newcommand{\D}{{\mathbb D}} 
\newcommand{\E}{{\mathbb E}} \newcommand{\F}{{\mathbb F}} 
\newcommand{\G}{{\mathbb G}} \newcommand{\HH}{{\mathbb H}} 
\newcommand{\I}{{\mathbb I}} \newcommand{\J}{{\mathbb J}} 
\newcommand{\M}{{\mathbb M}} \newcommand{\N}{{\mathbb N}} 
\renewcommand{\P}{{\mathbb P}} \newcommand{\Q}{{\mathbb Q}} 

\newcommand{\R}{{\mathbb R}} \newcommand{\T}{{\mathbb T}} 
\newcommand{\U}{{\mathbb U}} \newcommand{\V}{{\mathbb V}} 
\newcommand{\W}{{\mathbb W}} \newcommand{\X}{{\mathbb X}} 
\newcommand{\Y}{{\mathbb Y}} \newcommand{\Z}{{\mathbb Z}} 
\title[Arithmetic behaviour of Frobenius semistability of
 syzygy bundles for plane trinomial curves] 
{Arithmetic behaviour of Frobenius semistability of syzygy bundles for plane trinomial curves} 
\author{V. Trivedi} 
\address{School of Mathematics, Tata Institute of 
Fundamental Research, Homi Bhabha Road, Mumbai-400005, India} 
\email{vija@math.tifr.res.in} \date{}

\subjclass[2010]{13D40, 14H60, 14J60, 13H15} \keywords{Taxicab distance, 
Hilbert-Kunz multiplicity, characteristic~$0$, semistability, Frobenius semistability and strong semistability, Harder-Narasimahan filtration}
\begin{abstract}Here we consider the set of 
bundles $\{V_n\}_{n\in \N}$  associated to the 
plane trinomial curves  $k[x,y,z]/(h)$.
We prove that the 
Frobenius semistability behaviour of the  reduction mod $p$ of $V_n$  
is  a  function of the congruence class of $p$ modulo $2\lambda_h$ (an integer 
invariant associated to $h$). 

As one of the consequences of this, we prove  that if $V_n$ is semistable in $\Char~0$ 
then its reduction mod $p$
 is strongly semistable,  
for $p$ in a Zariski dense set of primes. Moreover,
   for any given finitely many  such semistable  bundles $V_n$, 
there is a common Zariski dense set of such primes. 

\end{abstract}

\maketitle\section{Introduction}
In this paper we discuss the problems  regarding 
 Frobenius semistability behaviour of a vector bundle on a nonsingular projective curve.

Recall that a vector bundle $V$ on a nonsingular projective curve 
$X$ is {\em semistable} if
for any subbundle $W\subset V$, we have $\mu(W) \leq \mu(V)$ where 
$\mu(W) = \deg~W/\rank~W $.
If $V$ is not semistable then it has the unique {\em Harder-Narasimhan filtration} 
$$0\subset V_1\subset \cdots \subset V_n =V \quad\mbox{such that}\quad 
\mu(V_1) > \mu(V_2/V_1)\cdots > 
\mu(V/V_{n-1}),$$
where $V_i/V_{i-1}$ is semistable.
In this case one defines $\mu_{max}(V) = \mu(V_1)$ and $\mu_{min}(V) = \mu(V/V_{n-1})$.
Though in characteristic $0$, the pull back of a semistable vector bundle under  a finite map
is semistable, the same is not always true in positive characteristics.  On the other hand, 
the definition of semistablity implies that, if $F:X\longto X$ is the Frobenius 
morphism, and if $F^{*}V$ is semistable, then so is $V$.
However, if $V$ is a semistable and such that 
$F^*V$ is not semistable
then by the results of Shepherd-Barron~[SB] (Corollary~$2^p$) and X.Sun~[S] (Theorem~3.1)
there is a bound on $\mu_{max}(F^{*}V)-\mu_{min}(F^{*}V)$ 
in terms of the genus  of the curve and the rank of the vector bundle $V$.

We say a bundle $V$ is strongly semistable if $F^{s*}V$ is semistable for all $s\geq 0$, 
where  $F^s$ is the $s$-th 
iterated Frobenius. Recall that unlike semistable bundles, 
strongly semistable bundles in $\Char~p>0$ behave 
like semistable bundles in $\Char~0$, in many respects. 
On the other hand there is a result of Langer~(Theorem~2.7 of [L]), which says that 
if $V$ is a vector bundle (in a fixed $\Char~p$) then there is $s_0>>0$ such that 
the HN filtration of  $F^{s_0*}V$ is the 
{\em strongly semistable HN filtration}, {\em i.e.},  
 there is $s_0>>0$ such that 
the HN filtration of $F^{s_0*}V$ consists of  strongly semistable subquotients.

Now, suppose $X$ is a nonsingular curve defined over a field of characteristic 
$0$ and $V$ is a vector bundle on $X$ and 
if $V_p$ denotes the  ``reduction mod $p$'' of $V$, then 
 reduction mod $p$ of the 
HN (Harder-Narasimhan) filtration of $V$ is the HN Filtration  of $V_p$, for $p>> 0$. This is 
a consequence  of the
openness of the semistability condition (see [Mar]). 
However such an openness condition does not hold for Frobenius semistability. 

For example, let 
 $V = Syz(x,y,z)$ be the syzygy bundle on $X=\mbox{Proj}~R$, where 
$R= k[x,y,z]/(x^4+y^4+z^4)$ of $\Char~p\geq d^2$ then, by [HM] and [T1] it follows that
$$\begin{array}{lcl} 
p\equiv\pm 1\pmod{8} &\implies & F^{s*}V \quad\mbox{is 
semistable for all}\quad s\geq 0\\
p\equiv \pm 3\pmod{8} & \implies  & V\quad\mbox{is semistable and}\quad F^*V\quad\mbox{has 
the HN filtration}\\
& & \sL \subset F^*V\quad\mbox{and}\quad \mu(\sL) = \mu(F^*V)+2.
\end{array}$$

 Note that if $V$ (in characteristic $0$) has semistable reduction mod $p$ for infinitely many primes $p$ then 
it is semistable in $\Char~0$ to begin with, due to the openness of the 
semistability property.

We look at the following questions.
\begin{enumerate}
\item If $V$ is a semistable vector bundle on $X$ defined over $\Q$ 
 then is $V_p$ (the reduction mod $p$) strongly semistable for $p$ 
in a Zariski dense set of $\Z$?
\item  If $s_0$ is a number such that 
$F^{s_0*}V_p$ has strong HN filtration, then can one describe such an $s_0$ 
in terms of the invariants of the curve $X$, for all but finitely many $p$?

\item Is the  
Frobenius semistablility 
behaviour ({\em i.e.}, the minimal number $s_0$ and the instability degree 
$\mu_{max}(F^{s_0}V_p)-\mu_{min}(F^{s_0}V_p))$  a function of the
congruence class of $p~~(\mbox{modulo})~~N$, for some integer invariant $N$ of the curve $X$,  
for all but finitely many $p$? (instead, we may ask if for some finite Galois extension 
$K$ of $\Q$, the Frobenius semistability of $V_p$ depends only on the splitting behaviour 
of $p$  in $\sO_K$ (the ring of integers), for all but finitely many $p$).

\end{enumerate}
Here, in this paper,  we look at the bundles which arise from  
the syzygy bundles $W_n$ of trinomial plane curves $C$ in $\P^2$, defined
by the short exact sequences 
$$0\longto W_n \longto \sO_C \oplus \sO_C\oplus \sO_C \longto 
\sO_C(n)\longto 0,$$
where the third map is $(s_1,s_2, s_3)\to (s_1x^n, s_2y^n, s_3z^n)$.
The bundle $W_n$ is  alternatively denoted by $Syz(x^n, y^n, z^n)$.   
 
Recall that if $V$ is a rank $2$ 
vector bundle on a nonsingular projective curve $X$ defined 
over a field of characteristic $p >0$ then  either (a) $V$ is strongly 
semistable, {\em i.e.}, $F^{s*}V$ is semistable for every $s\geq 0$, 
or (b) for some $s\geq 0$, $F^{s*}V$is not semistable, and hence it has the 
nontrivial HN filtration, namely 
$\sL\subset F^{s*}V$ such that $\sL$  is a line bundle with 
$\mu(\sL) > \mu(F^{s*}V)$.
Note that for such an $s$, the HN filtration of $F^{s*}V$ is the 
strong filtration and $\mu_{max}(F^{s*}V)-\mu_{min}(F^{s*}V) = 
2[\mu(\sL) - \mu(F^{s*}V)]$.

In this paper we answer the above questions and generalize the above result of 
Monsky for the set of vector bundles 
\begin{equation}\label{st}
\begin{array}{lcl}
S_{st} & = & \{V_n \mid V_n = \pi^*W_n, ~\quad~~
\pi:X\longto C\quad\mbox{the normalization of}\quad C,\\
& &  W_n\quad
\mbox{is a syzygy bundle 
of}~~~C,\quad C\in \{\mbox{trinomial curves}\},~~~ n\in \N\},
\end{array}
\end{equation} 
where by a trinomial  curve $C$ we mean $C = \mbox{Proj}~k[x,y,z]/(h)$, for a 
homogeneous irreducible  trinomial $h$.
If a trinomial curve $C$ is  nonsingular then  $V_n = W_n$.

Let $h$ be a  trinomial plane curve of degree $d$, then following Monsky [Mo2],
it is either {\em irregular} or {\em regular} (see beginning of section~(2)). 
For irregular trinomials, the following theorem settles all the above questions.

\vspace{5pt}
\begin{thm}\label{irr}Let $h$ be a irregular trinomial of degree $d$  and 
 let $r$ be the multiplicity of the irregular point 
 (note $r\geq d/2$). Then 
 for  all $n\geq 1$, 
\begin{enumerate}
\item $r = d/2$  implies that the bundle $V_n$ is  strongly semistable and 
\item $r>d/2$ implies that the bundle $V_n$ is  not semistable
 to begin with. Moreover
it has the HN filtration $\sL\subset V_n$ such that 
$\mu(\sL) = \mu(V_n) +(2r-d)^2n^2/4d$.
\end{enumerate}
In particular  the semistability behaviour of $V_n$ is independent of the characteristic 
$p$, 
(equivalently one can say that it depends on the single congruence 
class $p\equiv~~1~~(\mbox{mod}~~1)$). \end{thm}
\vspace{5pt}  

Given a  regular trinomial $h$, there are associated 
 positive integers $\lambda$ and $\lambda_h$ (see Notations~\ref{n3}).
 Following is the main result of this paper:

\begin{thm}Let $h$ be a regular trinomial of degree $d$ then for given $n\geq 1$, 
there is   a  well defined set theoretic map 
$$\Delta_{h,n}: \frac{(\Z/2\lambda_h\Z)^*}{\{1,-1\}}\longto 
\left\{0, \frac{1}{\lambda_h}, \frac{2}{\lambda_h}, \ldots, \frac{\lambda_h-1}{\lambda_h}
\right\}
\times \{0, 1, 2, \ldots, \phi(2\lambda_h)-1\} \bigcup \{(1, \infty)\}$$
 such that, given  
 $p\geq \max\{n, d^2\}$, we have 
$$p\equiv\pm l\pmod{2\lambda_h}~~~\mbox{and}~~~  
\Delta_{h,n}(l) = (1, \infty)
\implies V_n~~~\mbox{is strongly semistable and}$$
$$p\equiv\pm l~~(\mbox{mod}~{2\lambda_h})~~~\mbox{and}~~~
\Delta_{h,n}(l) = (t, s) \implies 
s~~~\mbox{is the least integer such that}~~~
F^{s*}V_n~~~\mbox{ is not semistable}$$
$$\mbox{and}~~~F^{s*}V_n~~~~\mbox{has the HN filtration}\quad
\sL \subset F^{s*}V_n\quad\mbox{with}\quad
\mu(\sL) = \mu(F^{s*}V_n) +\frac{\lambda}{2}(1-t).$$
\end{thm}

The existence of such a map has several consequences:

\vspace{5pt}
 \noindent~$(1)$\quad The Frobenius semistability behaviour of $V_n$, for a regular trinomial, 
is a function  on the congruence class
of $\pm p\pmod{2\lambda_h}$ (which are atmost 
$\phi(2\lambda_h)/2$ in number). 

In Section~$4$, we  compute $\Delta_{h,n}(1)$, for every $h$ 
and do more elaborate computations for 
symmetric (Definition~\ref{d5}) trinomials.  

\vspace{5pt}
 \noindent~$(2)$\quad In particular we deduce that (Theorem~\ref{c2})
 if  
$p\geq \max\{n, d^2\}$ then 
a semistable  bundle $V_n$ is always strongly 
semistable for  $p\equiv\pm 1\pmod{2\lambda_h}$, hence  given a
finite subset $\{V_{n_1}, \ldots, V_{n_s}\}$ of semistable  bundles
of $S_{st}$ (see~(\ref{st})),
the set of primes $p$, 
 for which every $V_{n_i}$ is strongly semistable, is a Zariski dense set 
(Corollary~\ref{r3}).

Moreover  $V_1$ over a regular trinomial is always semistable and 
hence 
\begin{enumerate}\item[(i)]  for a given  finite set of syzygy bundles $V_1$  of 
 regular  trinomial curves, there is a Zariski dense set of primes, 
for which  each of the bundles  is strongly semistable. 
On the other hand 
\item[(ii)] for any symmetric trinomial $h$ of degree $d\geq 4$ and $d\neq 5$, 
we show that  there is a Zariski dense set of primes for which $V_1$ is  not 
strongly semistable.
\end{enumerate}

\vspace{5pt}
 \noindent~$(3)$\quad The existence of such a map $\Delta_{h,n}$  
also implies that if there is one  prime 
$p\geq \max\{n, d^2\}$ such that $V_n$ is not strongly semistable 
then (i)\quad there is 
a Zariski dense set of primes for which $V_n$ fails to be 
strongly semistable, and infact (ii)\quad (Theorem~\ref{t5})
there is a Zariski dense set of primes for which the first 
Frobenius pull back $F^*V_n$ is not semistable. 

\vspace{5pt}
 \noindent~$(4)$\quad Since either (i) $V_n$ is strongly semistable or 
(ii) $F^{s*}V_n$ is not semistable for some
$0 \leq s< \phi(2\lambda_h)$, 
to check the strongly semistability of $V_n$,  ({\em i.e., to check the  
semistability of $F^{s*}V_n$, for every
$s\geq 0$}),  it is enough to check that 
$F^{s*}V_n$ is semistable for  $s = \phi(2\lambda_h)$. 

It would be interesting to know  if such properties as in $(1)$-$(4)$ 
hold in greater generality.

 Moreover, because of the bound on $s$ (Theorem~\ref{r*} and Remark~\ref{r6}), 
 for any given explicit trinomial curve $\mbox{Proj}~R$ given by $h$, we can compute 
$\Delta_{h,n}(l)$ (see Remark~\ref{r4}). Therefore
for any $p\geq \{n, d^2\}$ ($p =\Char~R$) 
 we get an effective algorithm to compute $e_{HK}(R,(x^n, y^n, z^n))$
 and the HN slopes for all the Frobenius pull backs of $V_n$.

We compute some concrete examples.
By Corollary~\ref{cvb1}, if $h$ is symmetric trinomial of degree $d$
then
it is trivial to check if the bundle $V_n$ is semistable or not,
 for all $p\equiv \pm 1(\mbox{mod}(2\lambda_h)$, $p>\max\{n, d^2\}$.

We give some examples, where $V_n$  
need not be strongly semistable
and have complicated  Frobenius semistability behaviour.
In particular we look at the Klein $d$-curve, 
$h = x^{d-1}y+y^{d-1}z+z^{d-1}x$. Let $d\geq 4$ be even then 
  Monsky's computation in [Mo2] gives   
$$\begin{array}{lcl}
p\equiv \pm(d-1) \pmod{2\lambda_h} & \implies &  
V_1~~\mbox{is semistable and}~~~ F^*V_1~~\mbox{is not semistable}\\
 & & \mbox{such that}~~~ \mu(\sL) = \mu(F^*V_1) + (d^2-3d)/2\end{array}.$$ 
In Corollary~(\ref{ctd3}), we prove, for  $3.2^{m-2} < d-1 < 3.2^{m-1}$ if 
$$\begin{array}{lcl}
p\equiv \lambda_h\pm{2} \pmod{2\lambda_h} & \implies &  
F^{m-1}V_1~~~\mbox{is semistable and}~~~ F^{m*}V_1~~~\mbox{is not semistable}\\
& & \mbox{and}~~~~ \mu(\sL) = \mu(F^{m*}V_1) + 
(d-2)\left(2[d-1-3.2^{m-2}]\right)+2.\end{array}$$

In this paper we crucially used an  old result of Monsky for plane trinomial 
curves which involves  the notion of {\em taxicab distance} (introduced in [H] and [HM]):

\vspace{5pt}

{\bf Theorem}~(Monsky)~~(see Theorem~\ref{tm} for a more precise version)\quad
Let $R = k[x,y,z]/(h)$, where $h$ is a regular trinomial of degree $d$ 
over a field of $\Char~k = p>0$.
Then $$e_{HK}(R, (x^n,y^n,z^n)) = \frac{3dn^2}{4} + 
\frac{1}{p^{2s}d}\left(\frac{\lambda(1-t_{pn})}{2}\right)^2,$$  
where, either $s = \infty$, or $s<\infty$ and 
$(1-t_{pn}) > 0$ with $t_{pn} = \mbox{Td}(p^stn)$.
  
\vspace{5pt}

We combine this with the result from [T1] which gave
 a dictionary between $e_{HK}(R, (x^n,y^n,z^n))$ and 
the Frobenius semistability behaviour of the syzygy bundle $V_n$.

\vspace{5pt}

{\bf Theorem}~~~(see Theorem~\ref{tv} for the more precise version)~~~ of [T1]):\quad 
If $p\geq \max\{n, d^2\}$ then, 
$s= \infty$ implies
that bundle $V_n$ is strongly semistable.
 If $0\leq s < \infty $ then it is the least number such that 
$F^{s*}V_n$ is not semistable.
Moreover, for  the HN filtration of 
$$0\subset \sL \subset F^{s*}(V_n),\quad~~~\mbox{we have}~~~ 
\mu(\sL) = \mu(F^{s*}(V_n)+ \frac{\lambda(1-t_{pn})}{2}.$$

To prove the main theorem~\ref{r*}, for a regular trinomial $h$,  we define a 
 set ${S_h}\subset (\Z/2\lambda_h\Z)^3$, which is a disjoint
union of eight sets $\{T_{ijk}\}_{ijk}$.

We consider the set $L_{odd} = \{(u_1,u_2,u_3)\in \Z^3\}$, (which was introduced in
([H] and [Mo2]) 
as the disjoint union of four sets $\{L_{odd}^\delta\}_\delta$.
 
For each $\delta$ and $l,~~n\geq 1$, we define a map (Lemma~\ref{l2})
$f^\delta_{l,n}:\N\cup\{0\} \longto (\Z/2\lambda_h\Z)^3$ and  
characterize the numbers  $s$ and $t_{pn}$ (given as in the above theorem of Monsky)
in terms of the set
$\bigcup_{\delta}(\mbox{Im}~(f^\delta_{l,n})\cap S_h)$: The integer 
 $s$ is the minimum element of the set 
and if $s\in \mbox{Im}~(f^\delta_{l,n})\cap T_{ijk}$ ($\delta$ and $(i,j,k)$ 
will be unique with this property),  then   
$f^\delta_{l,n}(s)$ and $(i,j,k)$ determine $t_{pn}$ for all 
$p\equiv \pm l \pmod{2\lambda_h}$.

The  very definition of $f^\delta_{l,n}$ implies that
 the map factors through $\Z/\phi(2\lambda_h)\Z$ and 
$f^\delta_{l,n} = f^\delta_{l+2\lambda_h,n}$.
This 
gives a well defined map $\Delta_{h,n}$ as in Theorem~\ref{r*}.

As a corollary, for all
$p\equiv \pm{1}\pmod{2\lambda_h}$, 
we get a simple expression for all  trinomials $h$ of degree $d$ (Corollaries
~\ref{chk1}~(1) and \ref{chk}~(2):
$$\begin{array}{lcl}
e_{HK}(k[x,y,z]/(h), (x,y,z)) & = & 3d/4 \quad\mbox{if}~~~ h~~~
\mbox{is a regular trinomial},\\
& = & 3d/4 + \frac{(2r-d)^2}{4d},\quad~\mbox{if}~~~h~~~\mbox{is a irregular trinomial},\\
&  & \mbox{where}\quad h \quad\mbox{has the point of multiplicity}\quad r\geq d/2.
\end{array}$$

In Corollary~\ref{c3}, when $h$ is a  
Klein~$d$-curve defined over a field of $\Char~p >0$ such that   
 $p\equiv \lambda_h\pm 2 \pmod{2\lambda_h}$, 
(as expected from the discussion above) we generate
a more complex one.

\begin{rmk}
As stated earlier, for the Fermat quartic, 
the function $\Delta_{h,1}$ is completely  known  by the result of  [HM]. 
For the Fermat curve ($h = x^d+y^d+z^d$) and Klein $d$-curve 
($h = x^{d-1}y+y^{d-1}z+z^{d-1}x$)   
Monsky~[Mo2] has computed $\Delta_{h,1}(d-1)$, for even $d$, 
 and $\Delta_{h,1}(\lambda\pm (2d-2))$ for odd $d>5$ where $h$ is a  Klein $d$-curve.

 Questions about Frobenius  semistablity of $V_n$ 
for the Fermat curve 
are  also studied extensively  in  works of  
Brickmann-Kaid, Brenner, Kaid, St\"{a}bler etc. (see the recent paper [BK] 
and references given there).
\end{rmk}
\section{Preliminaries}
Let $h\in k[x,y,z]$ be a homogeneous irreducible trinomial of degree $d$, {\it i.e.}, $h = M_1+M_2+M_3$
where $M_i'$ are monomials of degree $d$. 
 
By Lemma~2.2 of [Mo2], one can divide such an $h$ in two types:
\begin{enumerate}
\item $h$ is `irregular' if one or more of the points $(1,0,0)$, $(0,1,0)$
$(0,0,1)$ of $\P^2$
has multiplicity $\geq d/2$ on the plane curve $h$, 
 \item  $h$ is 
`regular', {\it i.e.}, 
 the exponents $e_1$ of $x$ in $M_1$, $e_2$ of $y$ in $M_2$ and 
$e_3$ of $z$ in $M_3$, respectively,  are all $>d/2$.

Moreover any regular $h$ is equivalent ({\it i.e.}, one equation is  obtained from 
the other equation by
 some permutation of $x$, $y$ and $z$) to one of the following
 \begin{enumerate}
\item Type~(I):\quad $h = x^{a_1}y^{a_2} + y^{b_1}z^{b_2} +z^{c_1}x^{c_2}$, 
where $a_1, b_1, c_1 > 
d/2$, (here $e_1 = a_1$, $e_2 = b_1$ and $e_3 = c_1$).
\item Type~(II):\quad $h = x^d + x^{a_1}y^{a_2}z^{a_3} + y^bz^c$, $a_2$, $c > d/2$, 
(here $e_1 = d$, $e_2 = a_2$ and $e_3 = c$).
\end{enumerate}
\end{enumerate}

Given a regular trinomial $h$,  Monsky defines a set of positive integers $\{\alpha, \beta, \nu, \lambda\}$ as follows:
$\alpha = e_1+e_2-d$, 
$\beta = e_1+e_3-d$ and $\nu = e_2+e_3-d$. Moreover
 $\lambda = \frac{1}{d}\det(A)$, 
where $A$ is a $3\times 3$ matrix formed from the exponents of $x$, $y$ and $z$ in 
$M_1$, $M_2$ and $M_3$. 

\begin{notations}\label{n1} In particular, given a regular trinomial $h$,   
we can associate positive integers
$\alpha, \beta, \nu, \lambda >0$ as follows:

\begin{enumerate}
\item Type~(I)~$h = x^{a_1}y^{a_2} + y^{b_1}z^{b_2} +z^{c_1}x^{c_2}$, we denote 
$$\alpha = a_1+b_1-d,~~ \beta = a_1+c_1-d,~~ \nu = b_1+c_1-d,~~ 
\lambda = a_1b_1+a_2c_2-b_1c_2.$$  
\item Type~(II)~ $h = x^d + x^{a_1}y^{a_2}z^{a_3} + y^bz^c$, 
 we denote
$$\alpha = a_2, \beta = c,
\nu = a_2+c-d\quad\mbox{and}\quad \lambda = a_2c-a_3b.$$
\end{enumerate}
  Moreover we denote 
$$t = (t_1, t_2, t_3) = (\alpha/\lambda, \beta/\lambda, \nu/\lambda).$$

\end{notations}

\begin{defn}\label{d1} We recall the following definition given in 
[HM] and [Mo2],  where
$p = {\rm char}~k > 0$:
 Let $L_{odd} = \{u=(u_1, u_2, u_3)\in \Z^3\mid 
\sum_i u_i~~~~\mbox{odd}\}$. 
For any $u\in L_{odd}$ and for $s\in \Z$ and $n\geq 1$,
the {\em taxicab distance} between the triples 
$p^stn = (p^st_1n, p^st_2n, p^st_3n)$, and $u$ 
is   
$\mbox{Td}(p^stn,u) = \sum_i|p^st_in-u_i|$.
 
They define $\delta^*(tn) = p^{-s}(1-\mbox{Td}(p^stn,u))$, and 
$s$ is the smallest integer such that $\mbox{Td}(p^stn,u) <1$, 
for some $u\in L_{odd}$. If there is no such pair then they define  
$\delta^*(tn) = 0$.
\end{defn}

Following is the crucial Theorem~2.3 of [Mo2]

\begin{thm}\label{tm}
Let $R = k[x,y,z]/(h)$, where $h$ is a regular trinomial of degree $d$.
Then $$e_{HK}(R, (x^n,y^n,z^n)) = \frac{3dn^2}{4} + 
\frac{\lambda^2}{4d}\left[\delta^*(tn)\right]^2 = \frac{3dn^2}{4} + 
\frac{\lambda^2}{4dp^{2s}}\left(1-\mbox{Td}(p^stn)\right)^2, $$  
where $\alpha$, $\beta$, $\nu$ and $\lambda$ are as in Notations~\ref{n1}.  
\end{thm}

We extend  the definition of Monsky to every integer $l$, as follows.

\begin{defn}\label{d4}
For an integer $l\geq 1$ we denote
 $\mbox{Td}(l^st) = \mbox{Td}(l^st,u)$, if there exists a $u\in L_{odd}$ such
that  $\mbox{Td}(l^st,u)<1$ (note that such a $u$ is unique if it exists).
\end{defn}

\begin{lemma}\label{l1}
\begin{enumerate}
\item The triple  $(\alpha, \beta, \nu)$ satisfies the  
triangle inequalities: 
$\alpha < \beta + \nu$,  
$\beta < \alpha +\nu $ and  
$\nu < \beta + \alpha $,  
 \item  and $2\lambda \geq\alpha +\beta+\nu$. Moreover  
\item the inequality $\mbox{Td}(l^{-s}tn, u)<1$ has no solution, for $s >0$  and 
$l\geq n$.
\end{enumerate}
\end{lemma}
\begin{proof}\quad The triangle inequalities of $(1)$ are obvious as pointed out in [Mo2].

\noindent~(2)~(i)\quad 
Let $\alpha, \beta, \nu, \lambda$ be the associated integers 
to the trinomial $h$ of type~(I). 
Then  $2\lambda < \alpha +\beta+\nu$ implies 
 $$b_1(a_1-c_2-1) + 
a_2c_2 + d/2 + a_2 < c_1.$$ But $a_1-c_2-1 \geq 0$. Now if $(a)$\quad  $a_1-c_2-1 \geq 1$ then 
$b_1(a_1-c_2-1) + a_2c_2 + d/2 + a_2
> d \geq c_1$, which is a contradiction.
 $(b)$\quad If $a_1-c_2-1 = 0$ then $a_1 = t+1$ and $c_2 =t$,
 where $d = 2t+1$ or 
$d = 2t$. Now 
$$b_1(a_1-c_2-1) + 
a_2c_2 + d/2 + a_2 = t^2+d/2 +t > t+1 = c_1,$$ 
which is again a contradiction.

\noindent~(2)~(ii)\quad Let $\alpha, \beta, \nu, \lambda$ be associated to the trinomial 
$h$ of type~(II).
Then 
$$2\lambda < \alpha +\beta+\nu \implies a_2c-a_3b < a_2+c-d/2 \implies 
(a_2-1)(c-1)-1+d/2 < a_3b$$
 which is 
not possible as $a_2-1\geq a_3$, $c-1\geq b$ and $d\geq 2$.

This proves part~$(2)$. 

\noindent~$(3)$\quad Let $t =  (t_1,t_2, t_3) = 
(\alpha/\lambda, \beta/\lambda, \nu/\lambda)$. Note that $ s > 0$ and $l \geq n$ 
implies 
$ 0\leq \lfloor t_1 n/l^s\rfloor\ = \lfloor \alpha n/l^s\lambda \rfloor < 1$. 
Let $u = (u_1, u_2, u_3) \in L_{odd}$. Hence 
for $u_1$ odd we have 
 $|t_1n/l^s- u_1| = 1-(\alpha n/l^s\lambda)$ and for $u_1$ even we have 
$|t_1n/l^s-u_1| =  \alpha n/l^s\lambda$. Similar assertions hold for $u_2$ and $u_3$.
(i)\quad If $u_1$, $u_2$ and $u_3$ are odd then 
$\mbox{Td}(l^{-s}tn, u)=3-(\alpha+\beta+\nu)n/l^s\lambda$. Therefore the 
existence of  a solution 
for 
$$\mbox{Td}(l^{-s}tn, u) <1 \implies 2\lambda l^s < (\alpha +\beta +\nu)n \implies 
2\lambda < (\alpha +\beta +\nu).$$
which is not possible by $(2)$.

\noindent~(ii)\quad Suppose only one of the $u_i's$ is odd.
Without loss of generality we assume that $u_1$ is odd then
$u_2$ and $u_3$ are even.   
 Now $\mbox{Td}(l^st, u) < 1$ if and only if $\beta+\nu < \alpha$, which contradicts 
$(1)$. This proves the lemma.
\end{proof}

\section{Main theorem}
Throughout this section $h$ denotes  a regular trinomial.

\begin{notations}\label{n3}
Let $\alpha, \beta, \nu, \lambda$ integers associated to $h$ 
  as in Notations~\ref{n1}.
Let $a= \mbox{gcd}(\alpha, \beta, \nu, \lambda)$. Then we denote 
$$\lambda_h= \frac{\lambda}{a},~~~\alpha_1 = \frac{\alpha}{a}, ~~~\beta_1 = 
\frac{\beta}{a},~~~\nu_1 = \frac{\nu}{a}.$$
\end{notations}

\begin{defn}\label{d3}
Let 
\begin{equation}\label{*}
\delta = (\delta_1, \delta_2, \delta_3) \in \{(1,1,1), (1,0,0), (0,1,0), 
(0,0,1)\} \in L_{odd}.\end{equation}

For given such $\delta $, We say  $u \in L^\delta_{odd}$ if   
$u = (2v_1+\delta_1, 2v_2+\delta_2, 2v_3+\delta_3)$, where $v_1, v_2, v_3$ 
are integers. 
 
Thus we can partition $L_{odd}$ into four disjoint sets
$$L_{odd} = \bigcup_{\delta} L_{odd}^\delta  = L_{odd}^{(1,1,1)} 
\cup L_{odd}^{(1,0,0)} \cup L_{odd}^{(0,1,0)} \cup L_{odd}^{(0,0,1)}.$$ 
\end{defn}

\begin{defn}\label{d2}
Let $R = (\Z/2\lambda_h\Z)^3$. 
We say the element $(w_1, w_2, w_3) \in \Z^3$ {\it represents} (or is the 
represnetative of) $w\in R$ if 
$(w_1, w_2, w_3) = w \pmod{(2\lambda_h \Z)^3}$ such that 
$0\leq w_1, w_2, w_3 < 2\lambda_h $.  

\vspace{5pt}

 We define 
$$S_h = T_{000} \cup T_{100} \cup T_{010} \cup T_{001} \cup T_{110}\cup 
T_{011} \cup T_{101}\cup T_{111} \subset R, $$
where, for $(i, j, k) \in \{0,1\}^3$,  
$$T_{ijk} = \{ w\in R\mid 2\lambda_h i+(-1)^iw_1+ 2\lambda_h j+(-1)^jw_2 + 
2\lambda_h k+(-1)^kw_3 < \lambda_h,\quad (w_1, w_2, w_3)~~\mbox{represents}~~ w\}.$$
For example 
$$T_{000} = \{ w\in R\mid w_1+ w_2 + 
w_3 < \lambda_h,\quad\mbox{where}~~ (w_1, w_2, w_3)~~\mbox{represents}~~ w\}
\quad\mbox{and} $$ 
$$T_{100} = \{ w\in R\mid 2\lambda_h - w_1+ w_2 + 
w_3 < \lambda_h,\quad\mbox{where}~~ (w_1, w_2, w_3)~~\mbox{represents}~~ w\},
\quad\mbox{etc.}.$$

Note that 
$$\{\mbox{the representatives of}~T_{ijk}\} \subset [i\lambda_h, (i+1)\lambda_h)\times 
[j\lambda_h, (j+1)\lambda_h)\times [k\lambda_h, (k+1)\lambda_h)\subset \Z^3.$$
In particular the set $S_h$ is a disjoint union of $\{T_{ijk}\}_{i,j,k \in \{0,1\}}$.
\end{defn}

\begin{lemma}\label{l2} For a given  $\delta $ as given in Equation~(\ref{*}) and 
for  given integers $n, l\geq 1$, let 
\begin{equation}\label{e1}f_{l,n}^{\delta}:\N\cup \{0\} \longrightarrow 
R= (\Z/2\lambda_h\Z)^3,\end{equation}
 given by
$s\to (l^s\alpha_1 n -\delta_1\lambda_h, l^s\beta_1 n -\delta_2\lambda_h, 
l^s\nu_1 n -\delta_3\lambda_h) = \lambda_h(l^stn-\delta)\pmod{(2\lambda_h \Z)^3}$ 
be a set theoretic map. Then 
\begin{enumerate}
\item for an integer $s\geq 0$, the element 
$f^\delta_{l,n}(s) \in S_h$ if and only if
$\mbox{Td}(l^stn, u) < 1$ has a solution for some $u\in L^\delta_{odd}$. Moreover,
\item in this case, $f^\delta_{l,n}(s)$ determines $\mbox{Td}(l^stn,u)$:
$$\mbox{Td}(l^stn) = 2(i+j+k)+(-1)^i\frac{w_1}{\lambda_h}+ (-1)^j\frac{w_2}{\lambda_h}+ 
(-1)^k\frac{w_3}{\lambda_h},$$
where $(w_1, w_2, w_3)$ is the representative of $f^\delta_{l,n}(s)$ and 
$(i,j,k)$ is the triple such that  $f^\delta_{l,n}(s) \in T_{ijk}$.
 \end{enumerate}
\end{lemma}
\begin{proof} Suppose $\mbox{Td}(l^stn, u) < 1$ has a solution for some $u\in L^\delta_{odd}$.
Then we have $u = (2v_1+\delta_1, 2v_2+\delta_2, 2v_3+\delta_3)$, for some integers
$v_1, v_2$ and $v_3$. Therefore 
 $\mbox{Td}(l^stn, u) <1$ implies 
\begin{equation}\label{e4}|l^s\alpha_1 n-\delta_1\lambda_h -2v_1\lambda_h|+
|l^s\beta_1 n-\delta_2\lambda_h-2v_2\lambda_h|+|l^s\nu_1 n-\delta_3\lambda_h-
2v_3\lambda_h| <\lambda_h\end{equation}
Let $(w_1,w_2, w_3)$ be the representative of $f_{l,n}^\delta(s)$. Then
$$(w_1,w_2, w_3) =  (l^s\alpha_1 n-\delta_1\lambda_h+2k_1\lambda_h, 
l^s\beta_1 n-\delta_2\lambda_h+2k_2\lambda_h, l^s\nu_1 n-\delta_3\lambda_h+
2k_3\lambda_h),$$
 for some integers $k_1$, $k_2$ and $k_3$.
Hence by Equation~(\ref{e4}), 
$$|\frac{w_1}{2\lambda_h} - (v_1+k_1)| + |\frac{w_2}{2\lambda_h}-(v_2+k_2)| + 
|\frac{w_3}{2\lambda_h}-(v_3+k_3)| < \frac{1}{2}.$$
Now
$$w_1 \in [0,\lambda_h) \implies  v_1+k_1 = 0\quad\mbox{and}\quad 
|\frac{w_1}{2\lambda_h}-
(v_1+k_1)| = \frac{w_1}{2\lambda_h}.$$
If $w_1 \in [\lambda_h, 2\lambda_h)$ then $v_1+k_1 = 1$ and 
$|\frac{w_1}{2\lambda_h}-(v_1+k_1)| = 1-\frac{w_1}{2\lambda_h}$.
In other words
$$ w_1 \in [i\lambda_h, (i+1)\lambda_h) \implies v_1+k_1 = i~~~\mbox{and}~~~ 
|\frac{w_1}{2\lambda_h}-(v_1+k_1)| = i+(-1)^i\frac{w_1}{2\lambda_h}.$$ Similar 
statements hold for $w_2$ and $w_3$.
Now Equation~(\ref{e4}) gives  
$$i+(-1)^i\frac{w_1}{2\lambda_h} + j+(-1)^j\frac{w_2}{2\lambda_h} + 
k+(-1)^k\frac{w_3}{2\lambda_h} < \frac{1}{2},$$
which implies $f_{l,n}^\delta(s) \in T_{ijk} \subset S_h$. 

Conversely,
let $f^\delta_{l,n}(s) \in S_h$ then there exists a unique 
$T_{ijk}$ such that $f^\delta_{l,n}(s) \in T_{ijk}$. Therefore $f^\delta_{l,n}(s)$ is 
represented by $(w_1, w_2, w_3)\in \Z^3$ such that
\begin{equation}\label{e3} 2\lambda_h i+(-1)^iw_1+ 2\lambda_h j+(-1)^jw_2 + 
 2\lambda_h k+(-1)^kw_3 < \lambda_h.\end{equation} 
Let 
$$(w_1, w_2, w_3) = (l^s\alpha_1 n -\delta_1\lambda_h +2\lambda_h k_1, 
l^s\beta_1 n -\delta_2\lambda_h +2\lambda_h k_2, l^s\nu_1 n -\delta_3\lambda_h +
2\lambda_h k_3).$$  
Then, by inequality~(\ref{e3}), we have 
$$\mbox{Td}(l^stn, u) = |\frac{l^s\alpha_1 n}{\lambda_h}-u_1|+ 
|\frac{l^s\beta_1 n}{\lambda_h}-u_2|+
|\frac{l^s\nu_1 n}{\lambda_h}-u_3| <1,$$
where 
$$u =  (\delta_1-2k_1+(-1)^{i}2i,~~ \delta_2-2k_2+(-1)^{j}2j,~~
\delta_3-2k_3+(-1)^{k}2k) \in L_{odd}^{\delta}.$$
This also proves that 
$\mbox{Td}(l^stn) = 2(i+j+k)+(-1)^iw_1/{\lambda_h}+ (-1)^jw_2/{\lambda_h}+ 
(-1)^kw_3/{\lambda_h}$, 
which proves part~(2) of the lemma and hence the lemma.
\end{proof}

\vspace{5pt}

\begin{thm}\label{r*} Let $h\in k[x,y,z]$ be a regular trinomial over a field of $\Char~p >0$. 
Consider the  set theoretic map
$$\Delta_{h,n}:\frac{(\Z/2\lambda_h\Z)^*}{\{1,-1\}}\longto 
\left\{\frac{1}{\lambda_h}, \frac{2}{\lambda_h}, \ldots, \frac{\lambda_h-1}{\lambda_h}
\right\}
\times \{0, 1, \ldots, \phi(2\lambda_h)-1\} \bigcup \{(1, \infty)\},$$
given by $l\to (\mbox{Td}(l), \mbox{Ds}(l))$, 
where $\mbox{Ds}(l)) = s\geq 0$ is the smallest integer, for which 
$\mbox{Td}(l^stn, u)< 1$ has a solution for some $u\in L_{odd}$ and 
$\mbox{Td}(l) := \mbox{Td}(l^stn,u)$.
If  there is no such $s$ then  $\Delta_{h,n}(l) = (1, \infty)$.
\begin{enumerate}
\item $\Delta_{h,n}$ is a well defined map.
\item $\Delta_{h,n} \equiv \Delta_{h,n+2\lambda_h}$.
\item Either $\mbox{Ds}(l) = \infty$ or  $\mbox{Ds}(l) < $ 
the order of the element $l$ in the group 
$(\Z/2\lambda_h\Z)^*$.
\item If $\Delta_{h,n}(l) = (t, s)$ for some $s < \infty$ and 
$ 1\leq s_1 $ divides $s$ then 
$\Delta_{h,n}(l^{s/s_1}) = (t, s_1)$. 
In particular if $\mbox{Im}~\Delta_{h,n} \neq \{(1, \infty)\}$ then 
$(t, 1) \in  \mbox{Im}~\Delta_{h,n}$, for some $t>0$.

\end{enumerate}
\end{thm}
\begin{proof} $(1)$\quad By Lemma~\ref{l2}, the inequality
$\mbox{Td}(l^stn, u)<1$ has a solution if and only if $f_{l,n}^{\delta}(s)\in S_h$, 
for some $\delta $ (if it does then $f_{l,n}^{\delta}(s)\in S_h$, for a unique $\delta$).
Hence
$\mbox{Ds}(l) = \min\{s'\mid s'\in (\bigcup_\delta \mbox{Im}(f^{\delta}_{l,n}))\bigcap S_h\}$

Let $$B = \left\{0, \frac{1}{\lambda_h}, \frac{2}{\lambda_h}, \ldots, 
\frac{\lambda_h-1}{\lambda_h}
\right\}
\times \{0, 1, \ldots, \phi(2\lambda_h)-1\} \bigcup \{(1, \infty)\},$$
 and let $ \Z_{\geq 0}\longto B$ be the map 
 given by $l\longto (\mbox{Td}(l), \mbox{Ds}(l))$. By the  definition of 
$f^{\delta}_{ln}$, it follows that 
$f_{l,n}^\delta(s)= f_{l+2\lambda_h,n}^\delta(s)$, for all $s\geq 0$. Therefore, 
by Lemma~\ref{l2}~(2), $Ds(l) = Ds(l+2\lambda_h)$ and $Td(l) = Td(l+2\lambda_h)$.
Hence  the above map factors through $\Z/2\lambda_h\Z\longto B$. which gives a well
defined map 
$(\Z/2\lambda_h\Z)^*\longto \Z/2\lambda_h\Z\longto B$.

Now let  $l'=2\lambda_h-l$ then $l'^s=2\lambda_h k+(-l)^s$, 
for some integer $k$.
If $s$ is even then $f_{l,n}^\delta(s) = f_{l',n}^\delta(s)$.

If $s$ is odd then $l'^s = 2\lambda_h k -l^s$.
Let $u=(u_1,u_2,u_3)\in L_{odd}^\delta $ such that 
$\mbox{Td}(l^stn, u) <1$ has a solution.
Then $$\mbox{Td}(l^stn,u) = |\frac{l^s\alpha_1 n}{\lambda_h}-u_1|+ 
|\frac{l^s\beta_1 n}{\lambda_h}-u_2|+|\frac{l^s\nu_1 n}{\lambda_h}-u_3| = 
\mbox{Td}(l'^stn,u') < 1,$$
where $u'= (2k\alpha_1 n-u_1,2k\beta_1 n-u_2, 2k\nu_1 n-u_3)\in L_{odd}^\delta$. 
This implies that for any $s\geq 0$, $f_{l,n}^\delta(s)\in S_h$  if and only if 
 $f_{l',n}^\delta(s)\in S_h$ and $\mbox{Td}(l^stn) = \mbox{Td}(l'^stn)$. Hence 
$(Td(l), Ds(l)) = (Td(l'), Ds(l'))$. This gives the well defined map 
$(\Z/2\lambda_h\Z)^*/\{1,-1\}\longto B$, which is $\Delta_{h,n}$.
 This proves assertion~(1) of the theorem.
  
\noindent~$(2)$\quad Since $f_{l,n}^\delta\equiv f_{l,n+2\lambda_h}^\delta $, assertion~(2) follows.

\noindent~$(3)$\quad If $Ds(l) <\infty$ then  $f_{l,n}^\delta(s) \in S_h$. for some $s\in\Z_{\geq 0}$.
Let order of $l$ in $(\Z/2\lambda_h\Z)^*$ be $t$.
 We can write $s= kt+r$, for some integers $k$ and $r$ such that 
$0\leq r <t$. Then $l^s = l^{kt}l^r = (2\lambda_h k_1+1)l^r$, for some $k_1\in \Z$.
This implies $f_{l,n}^\delta(s) = f_{l,n}^\delta(r)$, as $\alpha_1, \beta_1, \nu_1, 
\lambda_h$ are 
integers. Hence $Ds(l) \leq r < O(l)$.
This proves the assertion~(3).

\noindent~$(4)$\quad Note that $s$ is the minimal integer such that 
$f^\delta_{l,n}(s) \in S_h$ if and only if $s_1$ is the minimal integer such that
$f^\delta_{l^{s/s_1},n}(s_1) \in S_h$. Moreover 
$f^\delta_{l,n}(s)  = f^\delta_{l^{s/s_1},n}(s_1)$. Therefore 
$\Delta_{h,n}(l^{s/s_1}) = (t,s_1)$. This proves the assertion~(4) and 
  hence the theorem. \end{proof}

\begin{cor}\label{l5}Let $s\geq 0$ and $1\leq l < 2\lambda_h$ be  integers. 
Then 
$$\mbox{Td}(l^stn) = \mbox{Td}(p^stn)\quad\mbox{for}\quad p\geq n \quad\mbox{where}\quad  
p\equiv \pm l\pmod{2\lambda_h}.$$
Moreover, in that case 
$$\mbox{Td}(l^stn) =  \mbox{Td}((2\lambda_h-l)^stn) = \mbox{Td}(l^st(n+2\lambda_h))
= \mbox{Td}(p^stn).$$
\end{cor}
\begin{proof}It follows from Theorem~\ref{r*}.
\end{proof}

\begin{rmk}\label{r6}If $h$ is a regular trinomial as in Theorem~\ref{r*} then 
for given $n\geq 1$, we can further
reduce the number of such congruence classes:
Let $a_n = \mbox{g.c.d.}(\alpha n, \beta n, \nu n, \lambda)$. Let
 $\lambda_{h,n} = \lambda/a_n$. Then $\lambda_h$ is 
a multiple of $\lambda_{h,n}$.
Now,  respectively replacing
$\alpha_1 n, \beta_1 n, \nu_1 n, \lambda_h$ by 
$\alpha/a_n, \beta/a_n, \nu/a_n, \lambda_{h,n}$,  
in Lemma~\ref{l2} and Theorem~\ref{r*}, we get the same assertions,

In particular, 
$$\mbox{Im}~\Delta_{h,n}\subseteq
\left\{\frac{1}{\lambda_{h,n}}, \frac{2}{\lambda_{h,n}}, \ldots, 
\frac{\lambda_{h,n}-1}{\lambda_{h,n}}
\right\}
\times \{0, 1, \ldots, \phi(2\lambda_{h,n})-1\} \bigcup \{(1, \infty)\}.$$

\end{rmk}

\begin{rmk}\label{r4}
Given an explicit  trinomial $h$ of degree 
$d$ over a field of $\Char~p > 0$, let $p\equiv l\pmod{2\lambda_h}$ then  
we can compute $\Delta_{h,n}(l)$ in a effective way: 
Let $O(l)$ be the order of $l$ in $(\Z/2\lambda_h\Z)^*$ (infact can take $O(l)$ to be the order of 
$l$ in $(\Z/2\lambda_{h,n}\Z)^*$), 
we look for the first  $0\leq s \leq O(l)-1$, where $\mbox{Td}(l^stn, u)
= \sum_i|l^st_in-u_i|<1$ has a solution for some $u\in L_{odd}$. If there is such a 
solution then $\Delta_{h,n}(l) = (\sum_i|l^st_in-u_i|, s) = (t,s)$. Otherwise 
 $\Delta_{h,n}(l) = (1, \infty)$.\end{rmk}

\section{Computations of some values of $\Delta_{h,n}$}

We wiil see that 
$\Delta_{h,n}(l~~~\mbox{mod}~~{2\lambda_h})$ determines the Frobenius data 
(Lemma~\ref{l4}) of 
$V_n$ over the trinomial $h$, for $p\equiv \pm l \pmod{2\lambda_h}$ 
and also Hilbert-Kunz multiplicity (Theorem~\ref{thk}) 
of $k[x.y,z]/(h)$ with respect to the ideal $(x^n, y^n,z^n)$, we compute some of them.

\begin{thm}\label{ct1}  Let $h$ be a regular trinomial then 
\begin{enumerate}
\item For $n=1$, $\Delta_{h,n}(1~\mod 2\lambda_h) = (1, \infty)$. 
\item  In general, for $n > 1$, 
$$\begin{array}{lcl}
\mbox{either}\quad \Delta_{h,n}(1\mod 2\lambda_h) & = & (0, \infty),\\ 
\mbox{or}\quad  \Delta_{h,n}(1~\mod 2\lambda_h) & = & (\mbox{Td}(1), 0).
\end{array}$$
\end{enumerate}
\end{thm}
\begin{proof}
Since the order of the element $l =1$ is $1$ in $(\Z/2\lambda_h\Z)^*$.  
Assertion~(2) follows from Theorem~\ref{r*}~(3).
 
To prove Assertion~(1), it is enough to show that $\mbox{Td}(t, u)<1$ has 
no solution. Note that $\alpha, \beta, \nu <\lambda$. Therefore 
$\alpha/\lambda, \beta/\lambda, \nu/\lambda  < 1$.

Let  $u\in L_{odd}$ be a solution for $\mbox{Td}(t, u) < 1$. Then 
  $u_1$  odd implies   $u_1= 1$ which implies   $|\alpha/\lambda-u_1| = 
1-\alpha/\lambda$, and 
 $u_1$  even implies
$u_1= 0$ and 
$|\alpha /\lambda-u_1| =  \alpha/\lambda $.

\noindent~(i)\quad Suppose only one of the $u_i's$ is odd.
Without loss of generality we assume that $u_1$ is odd then
$u_2$ and $u_3$ are even.   
 Now $\mbox{Td}(t, u) = 1-\alpha/\lambda +\beta/\lambda +\nu/\lambda  < 1$ if and only if $\beta+\nu < \alpha$, which contradicts 
Lemma~\ref{l1}~(1).

\vspace{5pt}

\noindent~(ii)\quad Suppose $u_1$, $u_2$ and $u_3$ are all odd. Then  
$\mbox{Td}(t, 
u) = 1-\alpha/\lambda +1-\beta/\lambda +1-\nu/\lambda < 1$ implies $2\lambda < \alpha +\beta+\nu$, which is not true by 
Lemma~\ref{l1}~(2).
 This prove that $\mbox{Td}(t, u) < 1$ has no solution for any $s\in \Z$ and 
$u\in L_{odd}$. 
Hence $\Delta_{h,n}(1) = (1, \infty)$. This proves $(1)$.
\end{proof}

\subsection{Some Computations of $\Delta_{h,n}$ for 
symmetric trinomial curves}
\begin{defn}\label{d5} A trinomial curve $h$ of degree $d$ is {\em symmetric} if 
$h = x^{a_1}y^{a_2}+y^{a_1}z^{a_2}+z^{a_1}x^{a_2}$.
\end{defn}

\begin{rmk}\label{r5} A trinomial curve is symmetric if 
and only if $\alpha = \beta = \nu$. One can easily check that 
if $\mbox{Td}(l^stn,u) <1$ has a 
solution for some $(u_1,u_2,u_3)\in L_{odd}$ then $u_1 = u_2 = u_3$ and $u_1$ is odd.
\end{rmk}

\begin{cor}\label{ctd1} Let $h$ be a symmetric curve of degree $d$.
 For a given $n\geq 1$, we have 
\begin{enumerate}
\item $$\Delta_{h,n}(1~mod~2\lambda_h) = (\mbox{Td}(1), \mbox{Ds}(1)) = 
\left(3|m_1- \frac{\alpha n}{\lambda}|,~~~ 0\right),\quad\mbox{ 
if}\quad |m_1-\frac{\alpha n}{\lambda}|<1/3,$$
\item $$ \Delta_{h,n}(1~\mod~2\lambda_h) = (0,~~~ \infty)\quad \mbox{otherwise},$$
\end{enumerate}where   
 $m_1$ is one of the  nearest odd  integer to $\alpha n/ \lambda $. 
\end{cor}
\begin{proof}By Theorem~\ref{ct1}, 
it is enough to compute $\mbox{Td}(tn,u)$. 

If $\alpha n/\lambda $ is an even integer then one of the $u_i's$, say $u_1$, is equal to 
$\alpha n/\lambda\pm 1$, which implies $\mbox{Td}(tn) \geq 1$. 
So $Ds(1) = \infty$. On the other hand for any nearest odd integer $m_1$ to 
$\alpha n/\lambda $, we have $|m_1-\frac{\alpha n}{\lambda}| \geq 1$.
This proves the corollary for when $\alpha n/\lambda $ is an even integer.

Therefore  we can assume that 
$\alpha n/\lambda $ is not an even integer, 
and  $m_1 $ is the  unique nearest odd integer $m_1$.
 
Let  $u=(u_1,u_1,u_1)\in L_{odd}$ be a solution for $\mbox{Td}(tn,u)<1$ then 
$u_1 = m_1$ and hence 
 $\mbox{Td}(tn,u) = 3|m_1-\alpha n/\lambda|$, which is $ <1$ 
if and only if $|m_1-\alpha n/\lambda| < 1/3$. 

This implies $\Delta_{h,n}(1~{\rm mod}~2\lambda_h ) = (3|m_1-\alpha n/\lambda|,~~ 0)$
if $|m_1-\alpha n/\lambda|< 1/3$. Otherwise
 $\Delta_{h,n}(1) =  (1.~~\infty)$. \end{proof}

\begin{thm}\label{sns} Let $h$ be a symmetric curve of degree $d \geq 4$ and $d\neq 5$.
Then $\mbox{there is}\quad l'\in (\Z/2\lambda_h\Z)^*$ 
such that $\Delta_{h,1}(l') \neq (1, \infty).$
In fact there is $l'\in (\Z/2\lambda_h\Z)^*$ 
such that  
$$\begin{array}{lcl}
\Delta_{h,1}(l') & = \left({6}/{\lambda_h}, 1\right) 
& \mbox{if}\quad d\quad \mbox{is odd}\\
\Delta_{h,1}(l') & = \left({3}/{\lambda_h}, 1\right) & \mbox{if}\quad d\quad \mbox{is even 
and}\quad \lambda_h\quad \mbox{is even}\\
\Delta_{h,1}(l') & = \left(t, m\right) & \mbox{if}\quad d\quad \mbox{is even 
and}\quad \lambda_h\quad \mbox{is odd},
\end{array}$$
 where $1\leq m <\infty $ and $(t,m)$ is given as in Lemma~\ref{ltd1}.
\end{thm}
\begin{proof} $(1)$\quad Note that $d$ odd implies $\alpha, \lambda$ and 
hence $\alpha_1, \lambda_h$ are 
both odd.
 Since $\mbox{g.c.d}(\alpha_1, 2\lambda_h)=1$,
the map $(\Z/2\lambda_h\Z)^*\longto (\Z/2\lambda_h\Z)^*$
given by $l~~\mbox{mod}~~ 2\lambda_h \mapsto l\alpha_1~~\mbox{mod}~~ 2\lambda_h$ is 
bijetive.

Let  $l = \lambda_h+2$ then there is $l'\in (\Z/2\lambda_h\Z)^*$ such that 
$l = l'\alpha_1\pmod{2\lambda_h}$.
If $\lambda_h >6$ then 
$\Delta_{h,n}(l')=\left(3|\frac{\lambda_h+2}{\lambda_h}-1|, 1\right)$.

If $d> 12$ then $\lambda_h \geq (a_1-a_2)+{a_1a_2}/{(a_1-a_2)} >6$. One can check that 
for  $d=,7,9,11$
also  $\lambda_h>6$.
In particular 
$\Delta_{h,1}(l') = (6/\lambda_h, 1)$, if $d>5$ is odd.

$(2)~~(a)$\quad If $d$ and $\lambda_h$ are even then $\alpha_1$ is odd, which implies 
 $\mbox{g.c.d}(\alpha_1, 2\lambda_h)=1$.
Let  $l = \lambda_h+1$ then there is $l'\in (\Z/2\lambda_h\Z)^*$ such that 
$l = l'\alpha_1\pmod{2\lambda_h}$.
If $\lambda_h >3$ (which holds for $d\geq 4$) then $3|l'\alpha/\lambda -1| = 
3/\lambda_h <1$.
Therefore
$\Delta_{h,n}(l')=\left(3/\lambda_h, 1\right)$.

$(2)~~(b)$\quad Let $d$ be even and $\lambda_h$ be odd. 
In  Lemma~\ref{ltd1} for  $n=1$ we have  $m_1 =1$, which 
implies $|\alpha /\lambda -1|>1/3$. In particular, there is $1\leq m\infty$ such that 
$\Delta_{h,1}(\lambda_h\pm 2) = (t, m)\neq (1,\infty)$. 
This proves the theorem.
\end{proof}

\begin{lemma}\label{ltd1}Let 
$h$ be a symmetric trinomial  of even degree
such that   $\lambda_h$ is odd.    
Let   
 $m_1$ denote a  nearest odd  integer to $\alpha n/\lambda $. 
Then the number  $$|m_1-\alpha n/\lambda| \in \{1\} \bigcup 
\left(0,~~~\frac{1}{3}\right]
\bigcup_{m\geq 1} 
\left(1-\frac{4}{3.2^m},~~~1-\frac{2}{3.2^m}\right), \quad\mbox{and}$$
  \begin{enumerate}
\item  $$|m_1-\frac{\alpha n}{\lambda}| = 1\quad\mbox{or}\quad \frac{1}{3} \implies 
 \Delta_{h,n}(\lambda_h\pm 2)  = (1, \infty).$$
\item $$|\frac{\alpha n}{\lambda}-m_1| \in \left(0,~~~ \frac{1}{3}\right)\implies 
\Delta_{h,n}(\lambda_h\pm 2) = \left(3 |\frac{\alpha n}{\lambda}-m_1|,\quad 0\right).$$
\item  If $m\geq 1$ then  $$|\frac{\alpha n}{\lambda}-m_1| \in \left(1-\frac{4}{3.2^m}, ~~~
~~~1-\frac{2}{3.2^m}\right) \implies   \Delta_{h,n}(\lambda_h\pm 2) = 
(3.2^m\left||\frac{\alpha n}{\lambda}-m_1|-(1-\frac{1}{2^m})\right|,~~~m).$$
\end{enumerate}
\end{lemma}
\begin{proof}Assertions~(1) and (2) can be easily checked. Assertion~(3) can be checked
by dividing it into two cases:
$$(1)\quad |\frac{\alpha n}{\lambda}-m_1| \in \left(1-\frac{4}{3.2^m}, ~~~
~~~1-\frac{1}{2^m}\right]\quad\mbox{ and}\quad 
(2)\quad |\frac{\alpha n}{\lambda}-m_1| \in \left(1-\frac{1}{2^m}, ~~~
~~~1-\frac{2}{3.2^m}\right).$$
\end{proof}

\begin{cor}\label{ctd3}
Let  $h$ be symmetric trinomial of degree $d\geq 4$.
 If for $l \in (\Z/2\lambda\Z)^*$ there is an  
integer  $s\geq 0$ such that
$${3l^s}/{4}\leq {\lambda}/{\alpha} < {3l^s}/{2} \quad\mbox{then}\quad 
\Delta_{h,1}(l) = \left(3|{l^s\alpha}/{\lambda}-1|, s\right).$$

In particular if $h = x^{d-1}y + y^{d-1}z + z^{d-1}x$, where $d\geq 4$.
\begin{enumerate}
\item Suppose $d$  is an  even integer. 
 Then (such an  $m\geq 2$ always exists)
 $$3.2^{m-2} \leq  d-1 < 3\cdot 2^{m-1}~~{\implies}~~~ 
\Delta_{h,1}(\lambda\pm 2) = 
(3|1-{2^{m}\alpha}/{\lambda}|,~~~\{m\}).$$
\item Suppose $d$ is an odd integer then 
$$d=3\implies \Delta_{h,1}(\lambda\pm 2) = (1, \infty).$$ 
$$d=5\implies \Delta_{h,1}(\lambda\pm 2) = \left(\frac{6}{\lambda},
 ~~\{3\}\right).$$
  $$d\geq 7 \implies \Delta_{h,1}(\lambda\pm 2) = \left(\frac{6\alpha}{\lambda},~~\{1\}\right).$$
\end{enumerate}
\end{cor}
\begin{proof} First part of the corollary can be checked  by considering 
 two case (1) ${3l^s}/{4}\leq {\lambda}/{\alpha} < l^s $ and (2)
$l^s\leq {\lambda}/{\alpha} < {3l^s}/{2}$.

For the second part note that $d$ even implies $\alpha = d-2$ even and $\lambda = 
(d-1)(d-2)+1$ odd. Hence 
$\lambda\pm2 \in (\Z/2\lambda\Z)^*$. Now $\mbox{Td}(\lambda\pm 2)^st, u) < 1$ has a solution 
for some $u\in L_{odd}$ iff 
 $\mbox{Td}(\pm 2)^st,u') <1$ has a solution for some $u'\in L_{odd}$, 
as $(\lambda \pm 2)^s\alpha /\lambda  
\in 2^s\alpha/\lambda + 2\Z$. Also 
$$3.2^{m-2} \leq  d-1 < 3\cdot 2^{m-1} \iff 
(3/4)2^{m} \leq  d-1 < (3/2)2^{m}.$$ Hence the assertion follows from the 
first part of the corollary.

Now if $d$ is odd then for any $s\geq 0$, we have $(\lambda\pm 2)^s/\lambda = 
\mbox{odd integer} + (\pm 2)^s\alpha/\lambda$ as   $\lambda $ and $\alpha$ are both odd.
Now it is easy to check the rest.
\end{proof}

\section{semistability of syzygy bundles}

 Let $C =\mbox{Proj~R}$, where $R$ is an irreducible  plane curve given 
by a homogeneous  polynomial $h$ of degree $d$ over a field 
of characteristic $p$. 
Let $\pi:{X}\longto C$ be the 
normalization of $C$.
Consider the canonical sequence of $\sO_X$-modules 
$$0\longrightarrow W_n \longrightarrow \sO_C\oplus \sO_C\oplus\sO_C
\longrightarrow \sO_C(n) \longrightarrow 0,$$
where the third map  is given by $(s_1, s_2, s_3)\mapsto (s_1x^n, s_2y^n, s_3z^n)$.

We recall the following Theorem~5.3 of [T1], 
\begin{thm}\label{tv} Let $C$ be an irreducible curve of degree $d\geq 4$. Let 
$\pi: {X}\longrightarrow C$ be the normalization of $C$.   
Consider the canonical sequence of $\sO_X$-modules 
$$0\longrightarrow W_1 \longrightarrow \sO_C\oplus \sO_C \oplus \sO_C
\longrightarrow \sO_C(1) \longrightarrow 0.$$
Then 
\begin{enumerate}
\item either $e_{HK}(R,(x,y,z))= 3d/4$ and $V_1=\pi^*W_1$ is strongly semistable, or 
\item  
$$ e_{HK}(R,(x,y,z)) = \frac{3d}{4}+\frac{({\tilde l})^2}{4dp^{2s}},$$ 
\end{enumerate}

where ${\tilde l}$ is an integer such that $0< {\tilde l}\leq d(d-3)$ and 
$s\geq 0$ is the least 
number such that 
$F^{s*}(V_1)$ is not semistable.
Moreover, for the HN filtration of 
$$0\subset \sL \subset F^{s*}(V_1),\quad  
\mu(\sL) = \mu(F^{s*}(V_1))+\frac{{\tilde l}}{2}.$$
\end{thm}  

\begin{rmk}
If $V_1$ is replaced by $V_n$, then 
 the  same argument (see Lemma~4.7 and Corollary~4.11 of [T1],to justify the 
appearance of $n^2$ in the expression) shows
that 
\begin{equation}\label{**}e_{HK}(R, x^n, y^n, z^n) = \frac{3dn^2}{4} + 
\frac{({\tilde l})^2}{4dp^{2s}},\end{equation}
where $0\leq {\tilde l}\leq d(d-3)$ and $s$ is the least integer for which $F^{s*}V_n$ is 
not semistable. 
\end{rmk}

As we pointed out in [T1], the  bound on ${\tilde l}$ in terms of $d$ (which 
was obtained in [T1], using result from [SB] and [S]), 
gave   a dictionary between  
$s$ and ${\tilde l}$ and  $e_{HK}$ (although for $p>d(d-3)$.

For example in 1993 Hans-Monsky [HM] have explicity compute $e_{HK}$ for the plane curve
 $h= z^4+y^4+z^4$:
$$\begin{array}{lcl}
e_{HK}(k[x,y,z]/(h), (x,y,z)) & = & 3+ (1/p^2)~~~\mbox{if}~~~ p\equiv\pm 3\pmod{8}\\
 & = & 3~~~\mbox{if}~~~ p\equiv\pm 1\pmod{8}\end{array}.$$
Now, by Theorem~\ref{tv}, it is immediate that, for $p\geq 5$, ${\tilde l}=4$ and $s=1$, for 
$p\equiv\pm 3\pmod{8}$. This means $V$ is semistable but  $F^*V$ is not semistable.
On the other hand, it says that  
 $V_1$ is strongly semistable if  $p\equiv\pm 1\pmod{8}$.

\begin{thm}\label{thk}
Let $R = k[x,y,z]/(h)$, where $h$ is a regular trinomial of degree $d$ and $k$ is a field 
of characteristic $p>0$. If $p\geq n$ and  $p\equiv \pm l \pmod{2\lambda_h}$ 
then $$e_{HK}(R, (x^n,y^n,z^n)) = \frac{3dn^2}{4} + 
\frac{\lambda^2}{4d}\left[ \frac{1-t}{p^s}\right]^2,$$  
where $\Delta_{h,n}(l) = (\mbox{Td}(l), \mbox{Ds}(l)) = (t,s)$ is as  given in 
Definition~\ref{d1}.
\end{thm}
\begin{proof}If $p\geq n$, then by Lemma~\ref{l1}~~(3), 
$\mbox{Td}(p^{s}tn, u)<1$ has no solution for any $s <0$. Hence 
the minimum integer $s$, for which  $\mbox{Td}(p^stn, u) < 1$ has a solution for 
some $u\in L_{odd}$, is nonnegative.
Therefore, by Theorem~\ref{r*} and Corollary~\ref{l5},
$$\delta^*(\alpha n/\lambda,\beta n/\lambda, \nu n/\lambda ) = 
p^{-s}(1-t),$$ where 
$\delta^*({\alpha n}/{\lambda}, {\beta n}/{\lambda}, {\nu n}/{\lambda})$ 
is given as in Theorem~\ref{tm}.
Now the theorem follows from Theorem~\ref{tm}.
\end{proof}

The following Lemma explicitly relates $\Delta_{h,n}(l~~~\mbox{mod}~{2\lambda_h})$ and 
 the Frobenius 
semistability data of the syzygy bundle $V_n$ over $h$, for the set 
of primes $p\equiv \pm l\pmod{2\lambda_h}$, where $p\geq \{n, d^2\}$.
 
\begin{lemma}\label{l4}
Let $R = k[x,y,z]/(h)$, where $h$ is a regular trinomial of degree $d$ over an 
algebraically closed field of characteristic $p>0$. Let $p\geq \mbox\{n, d^2\}$ and 
let  $p\equiv \pm l\pmod{2\lambda_h}$. For  
$\Delta_{h,n}$ as in Theorem~\ref{r*},  
\begin{enumerate}
\item If  $\Delta_{h,n}(l) = (1, \infty)$ then  
$V_n$ is a strongly semistable bundle.
\item  if  $\Delta_{h,n}(l)  = (t,s)\neq (1, \infty)$ then 
$s$ is the least integer for which $F^{s*}V_n$ is not semistable. Moreover  
$F^{s*}(V_n)$ has the HN filtration 
$$0\subset \sL_n \subset F^{s*}(V_n),\quad\quad\mbox{where} \quad\quad 
\deg~\sL_n = \mu(F^{s*}V_n) + \frac{\lambda}{2}(1-t).$$
\end{enumerate}
\end{lemma}

\begin{proof} 
\noindent~$(1)$\quad If $\Delta_{h,n}(l) = (1, \infty)$, then 
$e_{HK}(R, (x^n,y^n, z^n)) = 3dn^2/4$ and therefore $V_n$ is strongly semistable.

\noindent~$(2)$\quad Let $\Delta_{h,n}(l) = (t,s)\neq (1, \infty)$.
By Theorem~\ref{thk} and Equation~(\ref{**}), we 
have
$$e_{HK}(R, (x^n, y^n, z^n)) =  
\frac{3dn^2}{4} + \frac{\lambda^2}{4d}\left[\frac{(1-t)}{p^s}\right]^2 
= \frac{3dn^2}{4} + \frac{{\tilde l}^2}{4dp^{2s_1}},$$
where $0\leq {\tilde l}\leq d(d-3)$ and $s_1\geq 0$ is the least integer for which 
$F^{s_1*}V_n$ is not semistable. Note, by Lemma~\ref{l1}, the integer $s\geq 0$.
This implies that 
$$\frac{{\tilde l}}{p^{s_1}} = \frac{\lambda}{p^{s}}(1-t).$$

Let $(u_1, u_2, u_3)\in L_{odd}$ such that $t = \mbox{Td}(p^stn,u) < 1$.
Therefore  $0<\lambda(1-t) <\lambda$. On the other hand
 $$\lambda(1-t) =   {a\lambda_h}(1-t) = a\left(\lambda_h-|p^{s}\alpha_1 n -
\lambda_h u_1| - 
|p^{s}\beta_1 n -\lambda_h u_2| - |p^{s}\nu_1 n -\lambda_h u_3|\right)
\in \Z.$$
This implies $\lambda(1-t) \leq\lambda$ is a positive integer.
This with the fact that   $0\leq {\tilde l} \leq d(d-3)$ implies that, for 
 $p\geq d^2$, we have 
$s_1 = s$ and hence ${\tilde l} = {\lambda}(1-t)$.
This proves the lemma.
\end{proof}

\vspace{5pt}
Racall that a trinomial curve is irregular or regular. For the irregular trinomials the 
 semistability behaviour is very  explicit and independent of 
the $\Char~p$ as stated in 
Theorem~\ref{irr}, a proof of which is along the same line as in 
  Theorem~4.9 of [T2].

In the light of Lemma~\ref{l4}, all the results in 
this section are immediate consequence of the results of the previous sections.

Following result gives the periodicity in the behaviour of 
$\{V_n\}_{n\in \N}$ where $V_n$ are syzygy bundles on a fixed trinomal $h$. 

\begin{thm}\label{t3}
For a regular trinomial defined over a field of characteristic 
$p$, if $p \geq n + 2\lambda_h $. then for any $s\geq 0$, 
\begin{enumerate}
\item the bundle 
$F^{s*}V_n$ is semistable if and only if $F^{s*}V_{n+2\lambda_h}$ is semistable. 
Moreover,  
\item $F^{s*}V_n$ has the HN filtration
$0\subset \sL_n\subset F^{s*}V_n$ if and only if $F^{s*}V_{n+2\lambda_h}$ 
has the HN filtration 
$0\subset \sL_{n+2\lambda_h}\subset F^{s*}V_{n+2\lambda_h}$  and in that case we have 
$ \deg\sL_{n+2\lambda_h} = \deg\sL_n - 3\lambda_h dp^s.$\end{enumerate}
\end{thm}
\begin{proof}Follows from Corollary~\ref{l5}.\end{proof} 

Following theorem implies that every semistable bundle $V_n$ over 
a trinomial is strongly semistable for a Zariski dense set of primes.
 
\begin{thm}\label{c2} Let $R = k[x,y,z]/(h)$ be a regular trinomial, where $k$ is an 
algebraically closed field of characteristic $p>0$. 
Let $p\equiv \pm 1\pmod{2\lambda_h}$ and $p \geq d^2$ then 
\begin{enumerate}
\item  $V_1$ is strongly 
semistable and  
\item if, in addition,  $p\geq n$ 
then 
\begin{enumerate}
\item either $V_n$ is strongly semistable or
\item $V_n$  itself is not semistable and 
has the HN filtration
$$0\subset \sL_n \subset \pi^*(V_n)~~~\mbox{where}~~~
 \deg\sL_n = -\frac{3nd}{2}+\frac{\lambda}{2}(1-t),$$ 
where $t=|\alpha n/\lambda -u_1|+|\beta n/\lambda -u_2|+|\nu n/\lambda -u_3|<1$ 
for a unique $(u_1, u_2,u_3)\in L_{odd}$.
\end{enumerate}
\end{enumerate}
\end{thm}

\begin{cor}\label{r3}If $V_{n_1}, \ldots, V_{n_s}$ are semistable syzygy bundles 
on  trinomials $h_1, \ldots, h_s$ respectively then they are all strongly semistable 
for  primes $p$ in a Zariski dense set\end{cor}
\begin{proof} Let $\lambda_h =1$ if $h$ is an  irregular trinomial. 
If  ${\tilde \lambda} = 
\mbox{l.c.m.}(\lambda_{h_1}, \ldots, \lambda_{h_s})$ then for 
$p\equiv \pm 1\pmod{2{\tilde \lambda}}$ the assertion holds.
\end{proof}

Following theorem asserts that to check the strong semistability property 
of a syzygy bundle $V_n$ 
over  a trinoimal $h$,  it is sufficient to
check the semistability of $V_n, F^*V_n, \ldots, F^{s*}V_n$, where $s <\phi(2\lambda_h)$.

\begin{thm}\label{t5}If $p \geq \max\{n, d^2\}$, then either 
\begin{enumerate}
\item \begin{enumerate}
\item $V_n$ is 
strongly semistable or 
\item there is $s<\phi(2\lambda_h)$ such that 
$F^{s*}(V_n)$ is not semistable. 

In fact if $p\equiv \pm l\pmod{2{\lambda_h}}$ then $F^{s*}(V_n)$ is not 
semistable for some $s<\mbox{order of}~~l$ in $(\Z/2\lambda_h \Z)^*$
\end{enumerate}

\item If there is a prime $p \geq \max\{n, d^2\}$ such that $V_n$ is not strongly 
semistable then there is a Zariski dense set for which $F^*V_n$ is not 
 semistable.
\end{enumerate}
\end{thm}
\begin{proof}Part~(1)~(a) and (b) follow from Theorem~\ref{r*}~(3).
For part~(2) suppose $p \geq \max\{n, d^2\}$ such that $V_n$ is not strongly 
semistable. There is $l\in (\Z/2\lambda_h\Z)^*$ such that 
   $p\equiv \pm l\pmod{2{\lambda_h}}$.
By Corollary~\ref{l5}, we have $\Delta_{h,n}(l) = (t,s) \neq (1, \infty)$.
Therefore there exists   $u\in L_{odd}$ such that $\mbox{Td}(l^stn, u)<1$. Now $l^s\in 
(\Z/2\lambda_h\Z)^*$ such that $\Delta_{h,n}(l^s) = (t, 1)$. Therefore for 
$p\equiv \pm l^s\pmod{2{\lambda_h}}$, the bundle $F^*(V_n)$ is not semistable.\end{proof}

By the following corollary it is trivial to 
check  if a syzygy bundle $V_n$, of a symmetric (see Definition~\ref{d5})
regular trinomial, is semistable or not.

\begin{cor}\label{cvb1}  
Let  $p\equiv \pm 1\pmod{2\lambda_h}$ and $p\geq \max\{n, d^2\}$ and let 
$h$ be a symmetric trinomial
 of degree $d$.
Let  
 $m_1$ denote  any of the nearest odd  integer to $\alpha n/\lambda $. 

\begin{enumerate}
\item If $|\frac{\alpha n}{\lambda}-m_1| \geq  \frac{1}{3}$ then 
$V_n$ is semistable (and hence strongly semistable),  and
\item if $|m_1-\alpha n/\lambda| < \frac{1}{3}$, then 
$V_n$ is not semistable 
and has the  HN filtration  
 $$0\subset \sL_n \subset V_n~~~~~\mbox{where}~~~~\deg\sL_n = 
-\frac{3nd}{2}+\frac{3\lambda}{2}\left[\frac{1}{3}-|\frac{\alpha n}{\lambda}-
m_1|\right].$$  
\end{enumerate}

\end{cor}

\begin{cor}\label{c41}Let $h$ be a symmetric trinomial
 of degree $d\geq 4$ but $d\neq 5$ then, for $p > d^2$, 
\begin{enumerate}
\item $V_1$ is strongly semistable for a Zariski dense set of primes and 
\item $V_1$ is  semistable but not strongly semistable  for a Zariski dense set of primes.
\end{enumerate}
\end{cor}
\begin{proof}Follows from Theorem~\ref{ct1}~(1) and  Theorem~\ref{sns}.\end{proof}

\begin{cor}\label{c4}Let $X$ be the plane curve given by $h = 
x^{d-1}y+y^{d-1}z+z^{d-1}x$
 where $k$ is a field of characteristic
$p \geq d^2$. 
 Let 
$$0\longrightarrow V \longrightarrow H^0(X,\sO_X(1))\tensor \sO_X
\longrightarrow \sO_X \longrightarrow 0,$$
be the canonical map.
\begin{enumerate}
\item  If  $p\equiv \pm 1\pmod{\lambda}$. Then $V$ is strongly semistable.
\item If   $p\equiv \pm 2\pmod{\lambda}$,  $d$ is even 
and 
\begin{enumerate}
\item If $d=4$ then 
$F^{*}V$ is semistable and the HN filtration of $F^{2*}V$ is given by 
$$0\subset \sL \subset F^{2*}(V)~~~\mbox{with}~~~\mu(\sL) = \mu(F^{2*}V) + 2$$

\item If $d\geq 6$ and $m\geq 2$ such that 
$(1)$\quad $3.2^{m-2}\leq d-1 < 2^{m}$ then 
$F^{m-1*}V$ is semistable and the HN filtration of $F^{m*}V$ is given by 
$$0\subset \sL \subset F^{m*}(V)~~~\mbox{with}~~~\mu(\sL) = \mu(F^{m*}V) + 
2\alpha(d-1-3.2^{m-2})+2, $$
$(2)$\quad if $2^{m}\leq d-1 < 3.2^{m-1}$ then 
$F^{m-1*}V$ is semistable and the HN filtration of $F^{m*}V$ is given by 
$$0\subset \sL \subset F^{m*}(V)~~~\mbox{with}~~~\mu(\sL) = \mu(F^{m*}V) + 
\alpha(3.2^{m-1}-(d-1))-1.$$

\end{enumerate}

\item If  $p\equiv \pm 2\pmod{\lambda}$ and  $d$ is odd then 
\begin{enumerate}
\item for $d\geq 7$, the bundle $V$ is semistable and for the HN filtration 
$0\subset \sL \subset F^*V$, we have 
$$\mu(\sL) = \mu(F^{*}V) + 
\left(\frac{\lambda-6\alpha}{2}\right).$$
\item If  $d=5$, $F^{2*}V$ is semistable and the HN filtration 
$0\subset \sL \subset F^{3*}V$, we have 
$$\mu(\sL) = \mu(F^{3*}V) + 
\frac{7}{2}.$$
\end{enumerate}
\end{enumerate}
\end{cor}
\begin{proof}Follows from Corollary~\ref{ctd3}.\end{proof}

\section{Hilbert-Kunz multiplicity}
Throughout this section $R= k[x,y,z]/(h)$, where $h$ is trinomial of degree $d$ 
and $k$ is a field of $\Char~k = p >0$.

\begin{cor}\label{chk1}Let $R= k[x,y,z]/(h)$,  where $h$ is  a trinomial of degree $d$. Let $n\geq 1$.
\begin{enumerate}\item If  $h$ is a irregular trinomial then
$$e_{HK}(R, (x^n,y^n,z^n)  =  \displaystyle{\frac{3dn^2}{4} + \frac{(2r-d)^2n^2}{4d}},$$
where $r$ is the multiplicity of the irregular point.
\item If $h$ is a regular trinomial then 
$$e_{HK}(R, (x^n,y^n,z^n))  =  \displaystyle{\frac{3dn^2}{4} + 
\frac{\lambda^2}{4dp^{2s}}(1-t)^2},$$
where $\lambda(1-t) \leq \lambda$ is a nonnegative integer and $0\leq s <\phi(2\lambda_h)$
and $t$ and $s$ are constant on the congreunce classes of $p~{mod}~~(2\lambda_h)$.
\end{enumerate}
\end{cor}

\begin{cor}\label{chk}If $h$ is a regular trinomial then 
\begin{enumerate}
\item for all $p\geq n+2\lambda_h$,
 $$e_{HK}(R, (x^{n+2\lambda_h},y^{n+2\lambda_h},z^{n+2\lambda_h})) = 
e_{HK}(R, (x^n,y^n,z^n)) + 3d(n\lambda_h+1).$$

\item if $p\equiv \pm 1\pmod{2\lambda_h}$  then  we have 
$$e_{HK}(R, (x,y,z)) = 3d/4,$$
\begin{enumerate}
\item 
$$
e_{HK}(R, (x^n,y^n,z^n))  =  \displaystyle{\frac{3dn^2}{4} + \frac{\lambda^2}{4d}(1-
t)^2},$$
if $t=|\alpha n/\lambda -u_1|+|\beta n/\lambda -u_2|+|\nu n/\lambda -u_3|<1$ 
has a solution for  (unique) $(u_1, u_2,u_3)\in L_{odd}$, otherwise

\item $$e_{HK}(R, (x^n,y^n,z^n))  =  \displaystyle{\frac{3dn^2}{4}}.$$
\end{enumerate}

\end{enumerate}\end{cor}
\begin{proof}Part~(1) follows from Corollary~\ref{l5}. 
Part~(2) follows  Theorem~\ref{ct1}.\end{proof}

\begin{cor}\label{c1} Let $h $ be symmetric trinomial of degree $d$.
Let  $m_1$ be one of  the nearest odd  integer to $\alpha n/\lambda $. 
\begin{enumerate}
\item Let $p\equiv \pm 1\pmod{2\lambda_h}$ then

\begin{enumerate}
\item $$|\frac{\alpha n}{\lambda}-m_1| \geq  \frac{1}{3} \implies  
e_{HK}(R,~~(x^n, y^n, z^n)) = \frac{3dn^2}{4},$$ 
\item
$$ |\frac{\alpha n}{\lambda}-m_1|  < \frac{1}{3} \implies 
 e_{HK}(R_p,~~(x^n, y^n, z^n)) = 
\frac{3dn^2}{4} +\frac{9\lambda^2}{4d}
\left[\frac{1}{3}-|\frac{\alpha n}{\lambda}-m_1|\right]^2. $$
\end{enumerate}
\item If $d$ is odd and $>5$ then there is $l'\in (\Z/2\lambda_h\Z)^*$ such that for 
 $p\equiv \pm l'\pmod{2\lambda_h}$
$$e_{HK}(R,~~(x, y, z)) = \frac{3d}{4} + \frac{\lambda^2}{4dp^2}\left[1-\frac{6}{\lambda_h}\right]^2.$$
\item If  $d\geq 4$ is even such that 
\begin{enumerate}
\item $\lambda_h$ is even then $e_{HK}(R,~~(x, y, z)) = \frac{3d}{4} + 
\frac{\lambda^2}{4dp^2}\left[1-\frac{3}{\lambda_h}\right]^2.$
\item If 
$\lambda_h$ is odd. Then 
$ e_{HK}(R_p, (x^n, y^n, z^n)) = 
\frac{3dn^2}{4} + \frac{\lambda^2}{4d}\left[\frac{1-t}{p^s}\right]^2,$
where $\Delta_{h,n}(\lambda\pm 2) = (t,s) \neq (1,\infty)$ (hence $0< t<1$ and 
$0\leq s<\infty$) is given as in Lemma~\ref{ltd1}.
\end{enumerate}
\end{enumerate}
\end{cor}

\begin{cor}\label{c3}Let $p\equiv\lambda \pm 2\pmod{2\lambda}$
and let  $h = x^{d-1}y + y^{d-1}z + z^{d-1}x$, where $d\geq 4$ (in this 
case $\lambda = \lambda_h$). 
\begin{enumerate} \item Suppose $d$ is an  even integer. 
 $$d=4 \implies e_{HK}(R, (x,y,z)) = 3+\frac{7}{p^4}.$$

 Let $d\geq 6$. Let  $m\geq 2$ such that  
 $3.2^{m-2} < d-1 < 2^{m}$. Then  

$$e_{HK}(R, (x,y,z)) = \frac{3d}{4}+\frac{4}{dp^{2m}}
\left[\alpha\left(d-1-3.2^{m-2}\right)+ 1\right]^2.$$

If $2^m < d-1 < 3\cdot 2^{m-1}$. Then  
$$e_{HK}(R, (x,y,z)) = \frac{3d}{4}+\frac{1}{dp^{2m}}\left[\alpha\left(3.2^{m-1}-(d-1)\right)-1\right]^2.$$

\item Suppose $d$ is odd then 
 $$d=5\implies e_{HK}(R, (x,y,z)) = \frac{3d}{4}+
\frac{1}{dp^6}\left[\frac{49}{4}\right]^2.$$

 $$d\geq 7 \implies e_{HK}(R, (x,y,z)) = \frac{3d}{4}+\frac{1}{4dp^2}
\left[(d-2)(d-7)+1\right]^2$$
\end{enumerate}
\end{cor}

\begin{rmk}If $h$ is a regular trinomai of degree  $d=3$ then it is 
an elliptic plane curve. 
Note that  $e_{HK}$ with respect to  the maximal ideal was first computed in  
[BC] and [Mo3].
Also on an  elliptic curve every semstable bundle is strongly semistable 
by [MR] (Theorem~2.1). 
\end{rmk}

\begin{rmk}\label{r2} For $R$ as
in Corollary~\ref{c3}, Monsky in [M1] had computed $e_{HK}$  
in the following situation:
\begin{enumerate}\item If $d\geq 4$ is even and $p\equiv \pm (d-1)\pmod{2\lambda}$ then 
$$e_{HK}(R) = \frac{3d}{4}+\frac{(d^2-3d)^2}{4dp^2}.$$
\item If $d\geq 5$ is odd and  $p\equiv \lambda\pm (2d-2)\pmod{2\lambda}$, 
$p\neq 2$ then
$$e_{HK}(R) = \frac{3d}{4}+\frac{(d^2-3d-3)^2}{4dp^2}.$$
\end{enumerate}\end{rmk}


\begin{thebibliography}{}

\bibitem[BK]{BK}{Brinkmann, D., Kaid, A.}, {\it 
Rank-2 syzygy bundles on Fermat curves and an application to Hilbert-Kunz functions}, 
Beitr. Algebra Geom. 57 (2016), no. 2, 321–342. 

\bibitem[BC]{BC}{Buchweitz,R., Chen,Q.}, {\it Hilbert-Kunz 
functions of cubic curves and surfaces,}
J. Algebra 197 (1997) 246-167.

\bibitem[H]{H}{Han, C.}, {\it The Hilbert-Kunz function of a diagonal hypersurfaces}, 
Ph.D. thesis, Brandeis University, 1991.


\bibitem[HM]{HM}{Han, C., Monsky, P.}, {\it Some surprising 
 Hilbert-Kunz functions}, Math. Z., { 214} (1993), no.~1, 119-135.

\bibitem[L]{L}{Langer, A.}, {\it Semistable sheaves in positive characteristic}, Ann. Math. 159 
(2004).

\bibitem[Mar]{Mar}{Maruyama, M.}, {\it Openness of a family of torsion
 free sheaves}, J. Math. Kyoto Univ. 16-3 (1976), 627-637.  

\bibitem[MR]{MR}{Mehta, V., Ramanathan, A.}, {\it
 Homogeneous bundles in characteristic p},
in {\it  Algebraic geometry -- open problems (Ravello, 1982)},
Lecture Notes in Math., 997, Springer, Berlin, 1983, 
 315 - 320.

\bibitem[Mo1]{Mo1}{Monsky, P.}, {\it The 
Hilbert-Kunz function}, Math. Ann. 263 (1983) 43-49. 


\bibitem[Mo2]{Mo2}{Monsky, P.}, {\it The 
Hilbert-Kunz multiplicity of an irreducible trinomial}, Journal of Algebra 304
 (2006) 1101-1107. 

 \bibitem[Mo3]{Mo3}{Monsky, P.}, {\it The Hilbert-Kunz function of a 
characteristic 2 cubic}, J. Algebra 197 1997, 268-277

\bibitem[SB]{SB}{Shepherd-Barron, N.I.}, {\it Semistability and reduction 
mod $p$}, Topology, { 37} (1998), no.~3, 659-664.

\bibitem[S]{S}{Sun, X.}, {\it Remarks on Semistability of
$G$-Bundles in Positive Characteristic},
Composition Mathematica, Vol. 119, (1999), 41-52.


\bibitem[T1]{T1}{Trivedi, V.}, {\it Semistability and Hilbert-Kunz  multiplicity 
for curves}, J. of Algebra, 284 (2005), 627-644, 
(arXiv:math/0402245v2 [math.AC] 21 Feb 2004).

\bibitem[T2]{T2}{Trivedi, V.}, {\it Strong semistability and Hilbert-Kunz multiplicity 
for singular plane curves}, Contemp. Math., 390, Amer.Math.Soc. 2005, 165-173.

\end{thebibliography}
\end{document}